\documentclass[final,1p,times,authoryear]{elsarticle}
\usepackage[utf8]{inputenc}
\usepackage[T1]{fontenc}
\usepackage[shortlabels]{enumitem}
\usepackage{amsmath}
\usepackage{amsthm}
\usepackage{amssymb}
\usepackage{amsfonts}
\usepackage{leftidx}
\usepackage{mathtools}
\usepackage[colorlinks=true, linkcolor=red, citecolor=blue, urlcolor=cyan]{hyperref}
\usepackage{todonotes}

\usepackage[nameinlink]{cleveref} 

\newtheorem{theorem}{Theorem}
\newtheorem{lemma}[theorem]{Lemma}
\newtheorem{algorithm}[theorem]{Algorithm}

\newtheorem{conjecture}[theorem]{Conjecture}

\newtheorem{definition}[theorem]{Definition}

\newtheorem{example}[theorem]{Example}
\newtheorem{remark}[theorem]{Remark}

\DeclareMathOperator{\lm}{lm}

\DeclareMathOperator{\LIdeal}{\ensuremath{\mathrm{L}}}

\DeclareMathOperator{\LM}{lm}
\DeclareMathOperator{\lc}{lc}
\DeclareMathOperator{\LC}{lc}
\DeclareMathOperator{\lt}{lt}
\DeclareMathOperator{\LT}{lt}
\DeclareMathOperator{\tail}{tail}

\DeclareMathOperator{\lcm}{lcm}

\newcommand{\IK}{\mathbb{K}}
\newcommand{\IN}{\mathbb{N}}
\newcommand{\IZ}{\mathbb{Z}}
\newcommand{\IQ}{\mathbb{Q}}

\renewcommand{\cite}{\citep}

\begin{document}

\begin{frontmatter}
\title{Computing Free Non-commutative Gr\"obner Bases over \texorpdfstring{$\mathbb Z$}{} with {\sc Singular:Letterplace}}

\author{Viktor Levandovskyy}
\address{Institut f\"ur Mathematik, Universit\"at Kassel, Germany}
\ead{levandovskyy@mathematik.uni-kassel.de}
\ead[url]{http://www.math.rwth-aachen.de/~Viktor.Levandovskyy/}

\author{Tobias Metzlaff}
\address{AROMATH, INRIA Sophia Antipolis M\'{e}diterran\'{e}e, Universit\'{e} C\^{o}te d'Azur, France}
\ead{tobias.metzlaff@inria.fr}


\author{Karim Abou Zeid}
\address{Lehrstuhl f\"ur Algebra und Zahlentheorie, RWTH Aachen University, Germany}
\ead{karim.abou.zeid@rwth-aachen.de}

\begin{abstract}
With this paper we present an extension of our recent ISSAC paper about computations of Gr\"obner(-Shirshov) bases over free associative algebras $\IZ\langle X\rangle$.
We present all the needed proofs in details, add a part on the direct treatment of the ring $\IZ/m\IZ$ as well as new examples 
and applications to e.g. Iwahori-Hecke algebras. 
The extension of Gr\"obner bases concept from polynomial algebras over fields to polynomial rings over rings allows to tackle numerous applications, both of theoretical and of practical importance. 
Gr\"obner and Gr\"obner-Shirshov bases can be defined for various non-commutative and even non-associative algebraic structures. We study the case of associative rings and aim at free algebras over principal ideal rings. We concentrate on the case of commutative coefficient rings without zero divisors (i.e.\ a domain). Even working over $\IZ$ allows one to do computations, which can be treated as universal for fields of arbitrary characteristic. By using the systematic approach, we revisit the theory and present the algorithms in the implementable form. We show drastic differences in the behavior of Gr\"obner bases between free algebras and algebras, which are close to commutative. 
Even the process of the formation of critical pairs has to be reengineered, together with implementing the criteria for their quick discarding.
We present an implementation of algorithms in the {\sc Singular} subsystem called {\sc Letterplace}, which internally uses Letterplace techniques (and Letterplace Gr\"obner bases), due to La Scala and Levandovskyy. Interesting examples and applications accompany our presentation.
\end{abstract}
\end{frontmatter}

\section*{Introduction}
\noindent
We present an extension of our recent ISSAC paper \cite{LMZ20} on the computations of Gr\"obner(-Shirshov) bases over free associative algebras like $\IZ\langle X\rangle$. In the extended version we have added and proved new results for non-commutative Gr\"obner bases over rings with zero-divisors using factorization and lifting techniques. The proof of \Cref{BuchbergerCriterionOverNonCommutativeRings} received substantial enhancements, since it is essential for the correctness of our algorithm. We added details to the proof of \Cref{SChainCrit} and showed the corresponding \Cref{GChainCrit}. New examples and applications are introduced in \Cref{IwahoriHeckeA3} and \Cref{Binomial}. Older examples are revisited and enhanced.
Lemma 2.9 from the original paper \cite{LMZ20} on the length bound $2d-1$ (where $d$ is the longest length of polynomials in a basis), which is needed to establish the finiteness of a strong Gr\"obner basis, is recasted as 
\Cref{finiteGB}. The latter is supported by the new \Cref{LemmaBoundRings}, which demonstrates, that the lowest length bound is  $3d-1$ indeed.

By 2010's the techniques based on Gr\"obner bases were well-established in the sciences and applications were widely known. A number of
generalizations of them to various settings have
been proposed and discussed. However, especially when it came to non-commutative and non-associative cases, generalizations of, in  particular, Gr\"obner bases, were often met with sceptical expressions like ``as expected'', ``straightforward'', ``more or less clear'' and so on. 
This is not true in general since
generalizations to various flavours of non-commutativity require deep analysis of procedures (and in case of provided termination, algorithms) based on intricate knowledge of properties of rings and modules over them.
Characteristically, in this paper we demonstrate in e.g.\  \Cref{BspProblemeUberZnonComm} and \Cref{ExProblemSPoly} how \emph{intrinsically different} Gr\"obner bases over $\IZ\langle X \rangle$ are even when compared with Gr\"obner bases over $\IQ\langle X \rangle$, not taking the commutative case into account. An example can illustrate this better than a thousand words: 
\begin{example}
Consider the set $F=\{2x,3y\}$. 
While taken over $\IZ\langle x, y \rangle$, 
it has a finite strong Gr\"obner basis $\{2x,3y,yx,xy\}$ with respect to any well-ordering. On the other hand, 
considered over $\IZ\langle x, y, z_1, \ldots, z_m \rangle$ for any $m\geq 1$, $F$ has an infinite Gr\"obner basis, which contains e.g.\  
$x z_i^k y$ and $y z_i^k x$ for any natural $k$.
\end{example}
\noindent
In his recent articles and in the book \cite{MoraBook4} Teo Mora has presented "a manual for creating your own Gr\"obner bases theory" over \emph{effective} associative rings. This development is hard to underestimate, as it presents a unifying theoretical framework for handling very general rings. 
In particular, we can address the Holy Grail of computational algebra,
that is the unified algorithmic treatment of finitely presented modules
over the rings like
\[
R = \left( \IZ \langle Y \rangle / J \right) \langle X \rangle/ I
\]
where $Y$ and $X$ are finite sets of variables,  $J$ is a two-sided ideal from the free ring $\IZ \langle Y \rangle$ and 
$I$ is a two-sided ideal from the associative 
ring $\left( \IZ \langle Y \rangle / J \right) \langle X \rangle$. The extension of $\left( \IZ \langle Y \rangle / J \right)$ with $X$ and $I$ can be iterated. In order to
compute within such a ring, it turns out to be enough to
have two-sided Gr\"obner bases over $\IZ \langle Z \rangle$ for a finite set of variables $Z$, with respect to -- among other -- block elimination orderings. Then, indeed, the concrete computation, still valid over $R$ will take place in $\IZ \langle Y \cup X \rangle/ (I+J)$. Furthermore, over the factor-algebra $R$ one needs left and right and two-sided (also called  bilateral) Gr\"obner bases for ideals and submodules of free bimodules.
We provide these components over fields and over rings $\IZ$.

\newpage

The theory of non-commutative Gr\"obner bases was developed by many prominent scientists since the Diamond Lemma of G.~Bergman \cite{BG}; notably important are the papers my Mora \cite{mora1,mora2}. Especially L.~Pritchard \cite{P96} proved versions of the PBW Theorem  and advanced the theory of bimodules, also over rings. On the other hand, procedures and even algorithms related to Gr\"obner bases in such frameworks are still very complicated. Therefore, when aiming at implementation, one faces the classical dilemma: generality versus performance. Perhaps the most general implementation, existing nowadays, is the {\sc JAS} system by H.~Kredel \cite{JAS, Kredel2015}. In our designs we balance the generality with the performance; based on {\sc Singular}, we utilize its long, successful and widely recognized  experience with data structures and algorithms in commutative algebra. Notably, the recent years have seen the in-depth development of Gr\"obner bases in commutative algebras with coefficients in principal ideal rings (O.~Wienand, G.~Pfister, A.~Fr\"uhbis-Kr\"uger, A.~Popescu, C.~Eder, T.~Hofmann and others), see e.g.\  
\cite{EderHofm19,EderPfisterPopescu16,EderPfisterPopescu18,Lichtblau12}. This required massive changes in the structure of algorithms; ideally, one has one code for several instances of Gr\"obner bases with specialization to individual cases.
In particular, the very generation of critical pairs and the criteria for discarding them without much effort were intensively studied. These developments were additional motivation for us in the task of attacking Gr\"obner bases in free algebras over commutative principal ideal rings, with $\IZ$ at the first place. Currently, to the best of our knowledge, no computer algebra system 
is able to do such computations. Also, a number of highly interesting applications wait to be solved: in studying representation theory of a finitely presented algebra (i.e.\ the one, given by generators and relations), computations over $\IZ$ remain valid after specification to \emph{any} characteristic and thus encode a universal information, see for example \Cref{IwahoriHeckeA3}.
In the system {\sc Felix} by Apel et al. \cite{Felix}, such computations were experimentally available, though not documented. In his paper \cite{Apel2000}, Apel demonstrates Gr\"obner bases of several nontrivial examples over $\IZ\langle X \rangle$, the correctness of which we can easily confirm now.

Our secret weapon is the \emph{Letterplace technology}
\cite{LL09,LL13,LSS13,LaScala14}, which allows the usage of commutative data structures at the lowest level of algorithms. We speak, however, in theory, the language of free algebras over rings, since this is mutually bijective with the language of Letterplace.

This paper is organized as follows: In the first chapter we establish the notations which are necessary when dealing with polynomial rings. Subsequently, in the second chapter we generalize the notion of Gr\"obner bases for our setup, present a theoretical version of Buchberger's algorithm and give examples to visualize significant differences compared to the field case or the commutative case. Implementation of Buchberger's algorithm depends on and benefits from the 
 gentle handling of critical pairs, which we will discuss in the third chapter. This is followed up by computational examples, applications and discussion on the implementational aspects. 

\section{Preliminaries}
\noindent
All rings are assumed to be associative and unital, but not necessarily commutative. 
We want to discuss non-commutative Gr\"obner bases over the integers $\IZ$. Equivalently one can take any commutative Euclidean domain $\mathcal{R}$. This concept can be extended to Euclidean rings. It was done in \cite{EderHofm19} for the commutative case with so-called annihilator polynomials.

We work towards an implementation and therefore we are interested in \emph{algorithms}, which \emph{terminate} after a finite number of steps. Since $\IZ\langle X \rangle$ is not Noetherian for $|X|\geq 2$, there exist finite generating sets whose Gr\"obner bases are infinite with respect to any monomial well-ordering. Therefore, a typical computation of a Gr\"obner basis is executed subject to the \emph{length bound} (where length is meant literally, applied to \emph{words} from the free monoid $\langle X \rangle$), specified in the input, and therefore terminates per assumption. Thus, we talk about \emph{algorithms} in this sense.

Our main goal is to obtain an algorithm to construct a Gr\"obner basis over such a ring, finding or adjusting criteria for critical pairs and setting up an effective method to implement Buchberger's algorithm in the computer algebra system \textsc{Singular}. The problem of applying the statements of commutative Gr\"obner basis over Euclidean domains and principal ideal rings, such as in \cite{EderPfisterPopescu18,EderPfisterPopescu16,Lichtblau12,MarkwigRenWienand15}, are divisibility conditions of leading monomials.

Let $X=\{x_1, \ldots,x_n\}$ denote the finite alphabet with $n$ letters. We set $\mathcal{P}=\mathcal{R}\langle X \rangle$, the free $\mathcal{R}$-algebra of $X$, where all words on $X$ form a 
basis $\mathcal{B}=\langle X \rangle$
of $\mathcal{P}$ as a free $\mathcal{R}$-module. The empty word in $\mathcal{B}$ as well as the neutral element of $\mathcal{R}$ are both denoted by $1$. From now on we say ``$\mathcal{B}$ is an $\mathcal{R}$-basis''.
Moreover, let $\mathcal{P}^e=\mathcal{P}\otimes_\mathcal{R} \mathcal{P}^{\mbox{opp}}$ be the free enveloping $\mathcal{R}$-algebra with basis $\mathcal{B}^e=\{u \otimes v \, \vert \, u, v \in \mathcal{B} \}$. The natural action $\mathcal{P}^e \times \mathcal{P} \to \mathcal{P}$, 
$(u \otimes v, t) \mapsto (u\otimes v)\cdot  t := u t v$ makes a bimodule $\mathcal{P}$ into a left $\mathcal{P}^e$-module.  
We call the elements of $\mathcal{B}$ \textbf{monomials} of $\mathcal{P}^e$. Note, that the tensor product is employed to facilitate uniformly computations with left, right and two-sided ideals and modules. 

Let $\preceq$ be a 
monomial well-ordering on $\mathcal{B}$.
With respect to $\preceq$,
a polynomial $f\in\mathcal{P}\setminus\{0\}$ 
has a \textbf{leading coefficient} $\LC(f)\in \mathcal{R}$,
a \textbf{leading monomial} $\LM(f)\in\mathcal{B}$ and a \textbf{leading term} $\LT(f)=\LC(f) \LM(f)\neq 0$. 
We denote by $\vert w \vert$ the length of the word $w \in \mathcal{B}$. The length corresponds to the total degree in the commutative case.

An ordering $\preceq$ is called \textbf{length-compatible}, 
if $u \preceq v$ implies $\vert u\vert \leq \vert v\vert$.
Every subset $\mathcal{G}\subseteq\mathcal{P}$ yields a two-sided ideal, the \textbf{ideal of leading terms} $\mbox{L}(\mathcal{G})=\langle \LT(f) \, \vert \, f\in \mathcal{G}\setminus \{0\}\rangle$.

The notions of leading coefficient, leading monomial and leading term carry over to an element $h\in\mathcal{P}^e$ by considering $h \cdot 1 \in \mathcal{P}$.

\begin{definition}\label{DefinitionOverlap}~\\
Let $u, v \in \mathcal{B}$. We say, that $u$ and $v$ have an \textbf{overlap}, if there exist monomials $t_1, t_2\in \mathcal{B}$, such that at least one of the four cases
\begin{align*}
(1)\,u t_1=t_2 v\hspace{.5cm}(2)\,t_1 u=v t_2\hspace{.5cm}(3)\,t_1 u t_2=v\hspace{.5cm}(4)\,u= t_1 v t_2
\end{align*}
holds. Additionally, we say, that $u$ and $v$ have a \textbf{nontrivial overlap}, if $(3)$ or $(4)$ holds, or if in $(1)$ or $(2)$ we have $\vert t_1\vert < \vert v\vert $ and $\vert t_2\vert < \vert u\vert$. In $(3)$, respectively $(4)$, we say, that $u$ \textbf{divides} $v$, respectively $v$ \textbf{divides} $u$.

The set of all elements, which are divisible by both $u$ and $v$, is denoted by $\textsc{cm}(u, v)\subseteq \mathcal{B}$ (\textsc{cm}: \emph{common multiple}).
The set of all elements which correspond to a minimal nontrivial overlap of $u$ and $v$ is denoted by $\textsc{lcm}(u, v)\subseteq \mathcal{B}$ (\textsc{lcm}:
\emph{least ...}), i.e.\ $t\in \textsc{lcm}(u, v)$ if and only if there exist $\tau_u, \tau_v \in \mathcal{B}^e$, such that $t=\tau_u \cdot u= \tau_v\cdot v$ represent nontrivial overlaps of $u$ and $v$, and for all $\tilde{t} \in \textsc{lcm}(u, v)$ with $\tilde{t}= \tau \cdot t $ for some $\tau \in \mathcal{B}^e$ we have $t=\tilde{t}$ and $\tau=1\otimes 1$. Should there only be trivial overlaps, then we set $\textsc{lcm}(u, v)=\emptyset$.
\end{definition}
\noindent
Fix a monomial well-ordering $\preceq$ on $\mathcal{B}$ and let $f, g\in \mathcal{P}$. Then if $\LM(g)$ divides $\LM(f)$, it follows that $\LM(g)\preceq \LM(f)$. To understand the set $\textsc{lcm}(u, v)\subseteq\mathcal{B}$, consider the following example.
\newpage
\begin{example}~\\
Let $u=xy$ and $v=yzx$ be words in the alphabet $\{ x, y, z \}$. There are four minimal overlaps of $u$ and $v$:

\begin{figure}[h!]
    \centering
    \begin{tabular}{cc|ccc}
        $x$ & $y$ & & & \\
        & & $y$ & $z$ & $x$ \\
    \end{tabular}
    \quad
    \begin{tabular}{c|c|cc}
        $x$ & $y$ \\
        & $y$ & $z$ & $x$ \\
    \end{tabular}
    \qquad
    \begin{tabular}{cc|c|c}
        & & $x$ & $y$ \\
        $y$ & $z$ & $x$ & \\
    \end{tabular}
    \qquad
    \begin{tabular}{ccc|cc}
        & & & $x$ & $y$ \\
        $y$ & $z$ & $x$ & & \\
    \end{tabular}
\end{figure}
\noindent
Among these four, two are trivial, namely $xyyzx$ and $yzxxy$, and two are nontrivial, namely $xyzx$ and $yzxy$. The set $\textsc{cm}(u,v)$ consists of all elements of $\mathcal{B}=\langle x, y, z \rangle$, which contain one of these four as a subword. The set of all minimal nontrivial overlaps is $\textsc{lcm}(u, v)=\{ xyzx, yzxy \}$.
\end{example}

\section{Non-commutative Gr\"obner Bases}
\noindent
Classically, a Gr\"obner basis for an ideal is a finite subset, whose leading terms generate the leading ideal. In the field case, this guarantees the existence of a so-called Gr\"obner representation, which will be recalled subsequently, and for any $f\in \mathcal{I}\setminus\{0\}$ it also guarantees the existence of an element $g\in \mathcal{G}$, such that $\LT(g)$ divides $\LT(f)$.

\begin{definition}~\\
Let $f, g \in \mathcal{P}\setminus\{0\}$, $\mathcal{G} \subseteq \mathcal{P}\setminus \{0\}$ a countable set and $\mathcal{I} \subseteq \mathcal{P}$ a two-sided ideal. 
We fix a monomial well-ordering $\preceq$ on $\mathcal{B}$.

$\mathcal{G}$ is called a \textbf{Gr\"obner basis} for $\mathcal{I}$, if $\LIdeal (\mathcal{I})\subseteq \LIdeal(\mathcal{G})$.

$\mathcal{G}$ is called a \textbf{strong Gr\"obner basis} for $\mathcal{I}$, if $\mathcal{G}$ is a Gr\"obner basis for $\mathcal{I}$ and for all $ f' \in \mathcal{I}\setminus\{0\}$ there exists $g' \in \mathcal{G}$, such that $\LT(g')$ divides $\LT(f')$ in $\mathcal{P}$.

We say that $f$ has a \textbf{strong Gr\"obner representation} w.r.t. $\mathcal{G}$, if $f=\sum_{i=1}^m h_i g_i$ with $m \in \IN, g_i \in \mathcal{G},  h_i \in \mathcal{P}^e\setminus\{0\}$ and there exists a unique $1\leq j \leq m$, such that $\LM(f)=$ $\LM(h_j g_j)$ and $\LM(f)\succ $ $\LM(h_i g_i)$ for all $i \neq j$.

If $\mathcal{R}$ is Euclidean, then we say that $g$ \textbf{lm-reduces} $f$, if $\LM(g)$ divides $\LM(f)$ with $\LM(f)=\tau \LM(g)$ for some $\tau \in \mathcal{B}^e$ and there are $a, b\in \mathcal{R}$, $a\neq 0$ and $\vert b \vert < \vert \LC(f) \vert$ in the Euclidean norm, such that $\LC(f)=a$\, $\LC(g)+b$. Moreover, the \textbf{lm-reduction} of $f$ by $g$ is given by $f-a\tau g$.
\end{definition}
\noindent
For our implementation we require lm-reductions, which are the key to obtain a remainder after division through a finite generating set $\mathcal{G}$ for an ideal and they are used in Buchberger's algorithm to construct a Gr\"obner basis from $\mathcal{G}$. Therefore, $\mathcal{R}$ is from now on an Euclidean ring. In this sense, the point of a Gr\"obner basis is to deliver a unique remainder when dividing through it. Since we operate in a polynomial ring of multiple variables, the expression ``reduction'' is more justified than ``division'' to describe a chain of lm-reductions. The outcome of such a reduction, or the ``remainder of the division'', is then known as a \textbf{normal form}.

The following normal form algorithm uses lm-reductions and can be compared to the normal form algorithms, which is used for algebras over fields in \cite{vikThesis}.

\newpage

\begin{algorithm}
\textsc{NormalForm}\\
\noindent\rule[4pt]{\linewidth}{.5pt}
\textbf{input}: $f\in \mathcal{R}\langle X \rangle\setminus\{0\}$, $\mathcal{G}\subseteq \mathcal{R}\langle X \rangle$ finite and partially ordered\\
\textbf{output}: normal form of $f$ w.r.t. $\mathcal{G}$\\
$01$:\hspace{0.2cm}$h=f$\\
$02$:\hspace{0.2cm}\textbf{while} $h\neq 0$ \textbf{and} $\mathcal{G}_h=\{g\in \mathcal{G} \mid g \mbox{ lm-reduces } h \} \neq \emptyset$ \textbf{do}\\
$03$:\hspace{0.7cm}choose $g \in \mathcal{G}_h$\\
$04$:\hspace{0.7cm}compute $a,b\in\mathcal{R}$ with: $a\neq 0$, $\LC(h)=a \LC(g)+b$ and $\vert  b \vert < \vert \LC(h) \vert$\\ 
$05$:\hspace{0.7cm}compute $\tau \in \mathcal{B}^e$ with: $\LM(h)=\tau  \LM(g)$\\
$06$:\hspace{0.7cm}$h=h-a \tau g$, the lm-reduction of $h$ by $g$\\
$07$:\hspace{0.2cm}\textbf{end while}\\
$08$:\hspace{0.2cm}\textbf{return} $h$\\
\noindent\rule[4pt]{\linewidth}{.5pt}
\end{algorithm}

Every normal form of the zero polynomial is zero. Termination (due to the usage of a well-ordering) and correctness (division with remainder) are completely analogous to the proof in \cite{vikThesis}. 
The output of the algorithm is not unique in general, for it depends on the choice of elements $g\in\mathcal{G}_h$ which are used for the reduction.

One can check, that the proof of the following theorem carries over verbatim from the commutative case in \cite{Lichtblau12}.

\begin{theorem}\label{CharacterizationOfGBNonCommEuclidRing} (Generalization of \cite[Theorem 9]{Lichtblau12})\\
Let $\mathcal{G} \subseteq \mathcal{P}\setminus\{0\}$ and $\{0\} \neq \mathcal{I} \subseteq \mathcal{P}$ an ideal. Then the following statements with respect to $\mathcal{G}$ and a fixed monomial well-ordering $\preceq$ are equivalent:
\begin{enumerate}
\item $\mathcal{G}$ is a strong Gr\"obner basis for $\mathcal{I}$.
\item Every $f\in \mathcal{I}\setminus\{0\}$ has a strong Gr\"obner representation.
\item Every $f\in \mathcal{P}\setminus\{0\}$ has a unique normal form after reduction.
\end{enumerate}
\end{theorem}
\noindent
An earlier non-commutative version has also been proven by Pritchard for ``weak'' Gr\"obner bases in \cite{P96}.

A strong Gr\"obner basis can be computed with Buchberger's algorithm using syzygy-relations between leading terms of generating polynomials. In the field case, the computation is done with S-polynomials. It is known from the commutative case over rings \cite{Lichtblau12}, that it does not suffice to take so called ``syzygy polynomials'' as in \Cref{Definition_S_and_G_Polynomial} to obtain a strong Gr\"obner basis. To see this, let $f=3x$, $g=2y$ and $\mathcal{I}=\langle f,g\rangle\subseteq \IZ\langle x, y \rangle$. Then every syzygy-polynomial of $f$ and $g$ is zero, but clearly $x y = f y - x g \in \mathcal{I}$ has a leading term which is neither divisible by $\LT(f)$ nor $\LT(g)$. Thus, $\{f, g\}$ is not a strong Gr\"obner basis for $\mathcal{I}$. The problematic polynomial $x y$ is constructed by looking at the greatest common divisor of the leading coefficients of $f$ and $g$.

\begin{definition}\label{Definition_S_and_G_Polynomial}~\\
Let $f, g \in \mathcal{P} \setminus\{0\}$ with $\textsc{lcm}(\LM(f), \LM(g))\neq \emptyset$ and choose $\tau_f, \tau_g \in \mathcal{B}^e$, such that $\tau_f \LM(f)=\tau_g \LM(g) \in \textsc{lcm}(\LM(f), \LM(g))$. Furthermore, let $a=\lcm(\LC(f),\LC(g))$ and $a_f, a_g \in \mathcal{R}$, such that $a=a_f \LC(f) = a_g \LC(g)$; let $b=\mbox{gcd}(\LC(f),\LC(g))$ and $b_f, b_f \in \mathcal{R}$, such that $b=b_f \LC(f)+b_g \LC(g)$\footnote{This is the B\'{e}zout identity for the leading coefficients.}. 
\newpage
In an Euclidean domain, the least common multiple 
and $a_f, a_g$ are determined uniquely up to a unit, 
$b$ is unique as a greatest common divisor, 
but the B\'{e}zout coefficients $b_f,b_g$ may not be unique, depending on the implementation of the Euclidean algorithm.\\

\noindent
We define a \textbf{first type S-polynomial of $f$ and $g$ with respect to $t$} as
\begin{align*}
\textup{spoly}_1^t(f, g):= a_f \tau_f f - a_g \tau_g g
\end{align*}
and a \textbf{first type G-polynomial of $f$ and $g$  with respect to $t$} as
\begin{align*}
\textup{gpoly}_1^t(f, g):= b_f \tau_f f + b_g \tau_g g.
\end{align*}
If such $\tau_f, \tau_g$ do not exist, we set the first type S- and G-polynomials to zero. Since two monomials may have several nontrivial overlaps, these $\tau_f,\tau_g$ are not unique. More precisely, this follows from the fact that $\mathcal{P}$ is not a unique, but a {\bf finite factorization domain} \cite{BHL-FFD}.
\end{definition}
\noindent
So far everything seems to work out as in the commutative case. We consider some examples to see that this impression is wrong.
\begin{example}\label{BspProblemeUberZnonComm}~\\
Let $f=2 x y, \, g= 3 y z \in \IZ[x, y, z]$, i.~e. in the commutative ring. We need compute an S-polynomial $3 f z-2 x g=0$ and a G-polynomial
\begin{align*}
\textup{gpoly}(f,g)=(-1) \cdot 2 x y \cdot z + 1 \cdot x \cdot 3y z=x y z.
\end{align*}
Since the latter does not reduce to zero, we add it to $\mathcal\{f, g\}$ and obtain a strong Gr\"obner basis $\{2xy, 3yz, xyz\}$ for $\mathcal{I}\subset \IZ[x, y, z]$.

Same computations need to be done for $f=2 x y, \, g= 3 y z \in \IZ\langle x, y, z \rangle$. But additionally, for every $w \in \mathcal{B}$ 
\begin{align*}
\textup{gpoly}'(f, g)=(-1) \cdot 2xy \cdot w \cdot yz + 1 \cdot xy \cdot w  \cdot 3yz=xywyz
\end{align*}
is also a G-polynomial of $f, g$ and for many different monomials $w$ the corresponding  polynomial $xywyz$ will be added to the basis. Note, that there is no finite Gr\"obner basis for $\mathcal{I}$ (since in particular $\{xy^{k}z : k\geq 2\}$ is a subset of any Gr\"obner basis). Thus we have to be satisfied with computing up to a fixed length bound for monomials, occuring in polynomials of the basis.
Note that in the case of a G-polynomial  we computed it in the canonical way, i.~e. by looking for a nontrivial overlap of $xy$ and $yz$. In the case of $\textup{gpoly}'$ we ignored this overlap. In the commutative case this is irrelevant, because $\textup{gpoly}(f, g)$ divides $\textup{gpoly}'(f, g)$. Furthermore, in the field case this is also irrelevant, because we do not need G-polynomials at all.
\end{example}
\noindent
Similar phenomena occur for S-polynomials.
\begin{example}\label{ExProblemSPoly}~\\
Let $f=2xy+x, \, g=3yz+z \in \IZ\langle x, y, z \rangle$. Then $\textup{spoly}(f, g)=3fz-2xg=xz$ is an S-polynomial of $f$ and $g$. Now consider
\begin{align*}
\textup{spoly}^w(f, g)=3fwyz-2xywg=3xwyz-2xywz
\end{align*}
for a monomial $w\in \mathcal{B}$. 
If $wy \neq yw$, the two monomials differ and then
\begin{align*}
(\textup{spoly}^w(f, g)-xwg)+fwz=0,
\end{align*}
meaning that this $\textup{spoly}$ reduces to zero. 
But for all $w$ such that $wy = yw$, i.~e.
for $w = y^k, k \geq 0$ we have
\begin{align*}
\textup{spoly}^w (f, g)= 3xy^k yz-2xy y^k z = 
xy^{k+1}z,
\end{align*}
which does not reduce any further and thus has to be added to the basis.
Furthermore, for $f=2xy+x$ we see that
\begin{align*}
\textup{spoly}^w(f, f)=fwxy-xywf=xwxy-xywx \neq 0
\end{align*}
is an S-polynomial of $f$ with itself which might not reduce any further depending on $w$ and we require $\LM(f)w\LM(f)\in \textsc{cm}(\LM(f),\LM(f))$, although it is not contained in $\textsc{lcm}(\LM(f),\LM(f))$. 

Thus, in general {\bf even principal ideals do not have finite strong Gr\"obner bases}! Such behavior of S-polynomials does not occur for non-commutative polynomials over fields.

Note, that we do not consider any further extensions of the leading monomials, meaning that the S- and G-polynomial corresponding to $t\in \textsc{lcm}(\LM(f),\LM(g))$ or $\LM(f) w \LM(g)$ make any further (trivial) overlap relations $\tau t$ or $\tau (\LM(f) w \LM(g))$ for $\tau \in \mathcal{B}^e$ redundant. Therefore, in the definition of $\textsc{lcm}(x, y)$ we stress the importance of the minimality.
\end{example}
\noindent
The previous example shows that we have to consider all possible S- and G-polynomials, but those are infinitely many. Moreover, the set $\textsc{cm}(\LM(f), \LM(g))$ contains too many elements that are redundant whereas the set $\textsc{lcm}(\LM(f), \LM(g))$ is too small. The following definition is made to classify two types of S- and G-polynomials, namely those corresponding to nontrivial overlap relations and those corresponding to trivial ones.
\begin{definition}\label{DefiFirstSecondTypeSGPoly}~\\
Let $f, g \in \mathcal{P}\setminus\{0\}$, $w\in \mathcal{B}$ and $a_f, a_g, b_f, b_g \in \mathcal{R}$ as in \Cref{Definition_S_and_G_Polynomial} with $\lcm(\LC(f),\LC(g)) = a_f \LC(f) = a_g \LC(g)$ and $\mbox{gcd}(\LC(f),\LC(g))=b_f \LC(f)+b_g \LC(g)$.\\
We define the \textbf{second type S-polynomial} of $f$ and $g$ w.r.t. $w$ as
\begin{align*}
\textup{spoly}_2^w(f, g):=a_f  f w \LM(g) - a_g \LM(f)w g
\end{align*}
and the \textbf{second type G-polynomial} of $f$ and $g$ w.r.t. $w$ as
\begin{align*}
\textup{gpoly}_2^w(f, g):=b_f  f w \LM(g) + b_g  \LM(f)w g.
\end{align*}
\end{definition}
\noindent
\begin{remark}\label{Remark_First_Second_Type_S_And_G_Polynomial}~\\
Clearly, it only makes sense to consider first type S- and G-polynomials if there is a nontrivial overlap of the leading monomials. However, as 
\Cref{BspProblemeUberZnonComm} shows, we always need to consider second type S- and G-polynomials. For any $w \in \mathcal{B}$ we have $ \LM(f) w \LM(g)  \in \textsc{cm}(\LM(f), \LM(g))$ and $ \LM(g) w \LM(f)  \in \textsc{cm}(\LM(f), \LM(g))$, which are distinct in general. Therefore, we need to consider both $\textup{spoly}_2^w(f, g)$ and $\textup{spoly}_2^w(g, f)$ and the same holds for second type G-polynomials. Also, note that the set of first type S- and G-polynomials is finite, because our monomial ordering is a well-ordering, whereas the set of second type S- and G-polynomials is infinite. Therefore, in this
context the need to compute up to a fixed length bound for the occuring monomials appears again in a natural way.

It is important to point out, that the elements $\tau_f, \tau_g$ are not uniquely determined. Take for example $f=2xyx+y, g=3x+1$. Then $t:=xyx=\LM(f)=xy \LM(g)\in\textsc{lcm}(\LM(f),\LM(f))$, but also $t=\LM(g)yx$ and thus $\textup{spoly}_1^t(f, g)=-3f+2gyx=2yx-3y$ and $(\textup{spoly}_1^t)'(f, g)=-3f+2xyg=2xy-3y$ are both first type S-polynomials with different leading monomials.
\end{remark}
\noindent 
The following algorithm uses Buchberger's criterion \Cref{BuchbergerCriterionOverNonCommutativeRings} as a characterization for strong Gr\"obner bases, which we will prove subsequently. It computes S- and G-polynomials of both kind up to a fixed length bound $d\in\IN$ and reduces them with the algorithm \textsc{NormalForm} in order to obtain a strong Gr\"obner basis up to length $d$ for an input ideal given by a finite generating set.

\begin{samepage}
\begin{algorithm}
\textsc{BuchbergerAlgorithm}\\
\noindent\rule[4pt]{\linewidth}{.5pt}
\textbf{input}: $\mathcal{I}=\langle f_1, \ldots, f_k  \rangle \subseteq \mathcal{R}\langle X \rangle$, $d\in \IN$, \textsc{NormalForm}\\
\textbf{output}: strong Gr\"obner basis $\mathcal{G}=\mathcal{G}_d$ up to length $d$ for $\mathcal{I}$\\
$01$:\hspace{0.2cm}$\mathcal{G}=\{f_1, \ldots, f_k \}$\\
$02$:\hspace{.2cm}$\mathcal{L}=\{ \textup{spoly}_1^t(f_i, f_j), \textup{gpoly}_1^t(f_i, f_j)\,\vert\,\forall\, t~^{*},i,j \}$\\
$03$:\hspace{.2cm}$\mathcal{L}=\mathcal{L}\cup\{ \textup{spoly}_2^w(f_i, f_j), \textup{gpoly}_2^w(f_i, f_j)\,\vert\, \forall\, w~^{**},i,j \}$\\
$04$:\hspace{0.2cm}\textbf{while} $\mathcal{L}\neq \emptyset$ \textbf{do}\\
$05$:\hspace{0.7cm}choose $h \in \mathcal{L}$ \\
$06$:\hspace{0.7cm}$\mathcal{L}=\mathcal{L}\setminus \{h\}$\\
$07$:\hspace{0.7cm}$h=\mbox{\textsc{NormalForm}}(h, \mathcal{G})$\\
$08$:\hspace{0.7cm}\textbf{if} $h \neq 0$ \textbf{then}\\
$09$:\hspace{1.2cm}$\mathcal{G}=\mathcal{G}\cup \{h\}$\\
$10$:\hspace{1.2cm}\textbf{for} $g\in \mathcal{G}$ \textbf{do}\\
$11$:\hspace{1.7cm}$\mathcal{L}=\mathcal{L}\cup\{ \textup{spoly}_1^t(g, h), \textup{gpoly}_1^t(g, h)\,\vert\,\forall\, t~^{*} \}$\\
\hphantom{$30$:}\hspace{1.7cm}$\mathcal{L}=\mathcal{L}\cup \{ \textup{spoly}_1^t(h, g), \textup{gpoly}_1^t(h, g) \,\vert\,\forall\, t~^{*}\}$\\
\hphantom{$30$:}\hspace{1.7cm}$\mathcal{L}=\mathcal{L}\cup \{ \textup{spoly}_2^w(g, h), \textup{gpoly}_2^w(g, h) \,\vert\,\forall\, w~^{**}\}$\\
\hphantom{$30$:}\hspace{1.7cm}$\mathcal{L}=\mathcal{L}\cup \{ \textup{spoly}_2^w(h, g), \textup{gpoly}_2^w(h, g) \,\vert\,\forall\, w~^{**}\}$\\
$12$:\hspace{1.2cm}\textbf{end do}\\
$13$:\hspace{0.7cm}\textbf{end if}\\
$14$:\hspace{0.2cm}\textbf{end while}\\
$15$:\hspace{0.2cm}\textbf{return} $\mathcal{G}$\\
\noindent\rule[4pt]{\linewidth}{.5pt}
$~^{*}$ $t\in \textsc{lcm}(\lm(\bullet_1),\lm(\bullet_2))$, such that $\vert t \vert <d$\\
$~^{**}$ $w\in\mathcal{B}$, such that $\vert \LM(\bullet_1)\vert+ \vert w \vert + \vert \LM(\bullet_2)\vert < d$\\
The monomials $t$, which satisfy $~^{*}$, come from pairs of type $\tau_{\bullet_1},\tau_{\bullet_2} \in \mathcal{B}^e$. Those pairs are not unique and so all first type S- and G-polynomials w.r.t. $t$ are computed.
\end{algorithm}
\end{samepage}

For the algorithm to terminate we need the set $\mathcal{L}$ to eventually become empty. This happens, if and only if after finitely many steps every S- and G-polynomial based on any combination of leading terms has normal form zero w.r.t $\mathcal{G}$, i.e.\ there exists a chain of lm-reductions, such that the current S- or G-polynomial reduces to zero. However, lm-reductions only use polynomials of smaller or equal length and all of these are being computed. Therefore, the algorithm terminates.

For the correctness of the algorithm we still need a version of Buchberger's criterion. More precisely, we want $\mathcal{G}$ to be a Gr\"obner basis for $\mathcal{I}$, if and only if for every pair $f, g \in \mathcal{G}$ all their S- and G-polynomials reduce to zero. Moreover, we only want to consider first and second type S- and G-polynomials, i.e.\ only use $t\in \textsc{cm}(\LM(f), \LM(g))$, such that one of the following four cases
\begin{align*}
    (1)\,t=\LM(f) t_f'=t_g \LM(g)\hspace{.5cm}(2)\,t=\LM(f)=t_g \LM(g) t_g'\\
    (3)\,t=t_f \LM(f) = \LM(g) t_g'\hspace{.5cm}(4)\,t=t_f \LM(f) t_f'=\LM(g)
\end{align*}
holds
for $t_f, t_f', t_g, t_g' \in \mathcal{B}$. This excludes all cases where $t$ is not minimal, i.e.\ $t=\tau t' $ for $\tau\in \mathcal{B}^e$ and $t'$ satisfying one of the above four cases. Pritchard has proven in \cite{P96}, that for a generating set of the left syzygy module (which is not finitely generated in general) we may 
 use only minimal syzygies.

\begin{lemma}\label{BuchbergerCriterionOverNonCommutativeRings} (Generalization of \cite[Theorem 10]{Lichtblau12})\\
Let $\mathcal{G}\subseteq \mathcal{P}\setminus\{0\}$. Then $\mathcal{G}$ is a strong Gr\"obner basis for $\langle \mathcal{G} \rangle$, if and only if for every pair $f, g\in \mathcal{G}$ their first and second type S- and G-polynomials reduce to zero w.r.t. $\mathcal{G}$.
\begin{proof}
The idea of the proof goes back to \cite{Lichtblau12}; we only need to show the ``if'' part. Let $0 \neq f\in  \langle\mathcal{G}\rangle  =:\mathcal{I}$ with $f=\sum_i h_i g_i$ for some $h_i \in \mathcal{P}^e$ and $g_i\in\mathcal{G}$. We set $t:= \mbox{max}(\LM(h_i g_i))$ and $M:=\{i\in \IN \mid \LM(h_i g_i)=t\}$. Clearly $\LM(f)\preceq t$ and we may assume that there is no other representation of $f$ where $t$ is smaller. Without loss of generality let $M=\{1, \ldots, m\}$. Moreover, since the Euclidean norm induces a well-ordering,  we can choose a representation where $\sum_{i=1}^m \vert \LC(h_i) \LC(g_i)\vert $ is minimal w.r.t. $t$. If $M$ contains exactly one element, then $t=\LM(f)$ and we have a strong standard representation of $f$ w.r.t. $\mathcal{G}$. Suppose otherwise that $\mbox{card}(M)>1$. Then $t\succeq \LM(f)$. Note that $t=\LM(h_i g_i)=\LM(h_i )\LM( g_i)$ for $i\leq m$. Then there exist monomials $t_1, t_1', t_2, t_2' \in X$, such that $t=t_1 \LM(g_1) t_1'=t_2 \LM(g_2) t_2'$. This induces an overlap relation of the leading monomials, because then there exist $s_1, s_1', s_2, s_2' \in X$, such that for one of the possibilities
\begin{align*}
    T:=\LM(g_1) s_1'=s_2 \LM(g_2) \hspace{1cm} T:=\LM(g_1)=s_2 \LM(g_2) s_2'\\
    T:=s_1 \LM(g_1) = \LM(g_2) s_2' \hspace{1cm} T:=s_1 \LM(g_1) s_1'=\LM(g_2)
\end{align*}
we obtain $t=\tau T$ for some $\tau\in \mathcal{P}^e$. Moreover, let $\tau_1, \tau_2$ result from $s_1, s_1', s_2, s_2'$, such that $\tau_1 T = \LM(g_1), \tau_2 T= \LM(g_2)$. Furthermore, let 
\begin{align*}
a_1:=\dfrac{\lcm(\LC(g_1), \LC(g_2))}{\LC(g_1)},\hspace{.3cm} a_2:=\dfrac{\lcm(\LC(g_1), \LC(g_2))}{\LC(g_2)}
\end{align*}
and $d:= \mbox{gcd}(\LC(g_1), \LC(g_2))=b_1 \LC(g_1)+b_2 \LC(g_2)\in \mathcal{R}$ (the Bézout identity for the leading coefficients). Now if $T$ corresponds to a nontrivial overlap, then we can compute $\textup{spoly}_1^T (g_1, g_2)$, $\textup{gpoly}_1^T (g_1, g_2)$ or $\textup{spoly}_1^T (g_2, g_1)$, $\textup{gpoly}_1^T (g_2, g_1)$. Otherwise, there exists a $w\in \mathcal{B}$, such that $T=\LM(g_1) w \LM(g_2)$ or $T=\LM(g_2) w \LM(g_1)$. In this case we are interested in $\textup{spoly}_2^w (g_1, g_2)$, $\textup{gpoly}_2^w (g_1, g_2)$ or $\textup{spoly}_2^w (g_2, g_1)$, $\textup{gpoly}_2^w (g_2, g_1)$. By definition
\begin{align*}
\begin{array}{ll}
            & \textup{spoly}(g_1, g_2):=a_1 \tau_1 g_1 - a_2 \tau_2 g_2 \\
\hbox{and}  & \textup{gpoly}(g_1, g_2):=b_1 \tau_1 g_1 + b_2 \tau_2 g_2
\end{array}
\end{align*}
are first or second type S- and G-polynomials and $\LM(h_1)=\tau \tau_1$, $\LM(h_2)=\tau \tau_2 \in\mathcal{B}^e$. Choose $a, b\in \mathcal{R}\setminus \{0\}$, such that $\LC(h_1)\LC(g_1)+\LC(h_2) \LC(g_2)=ad$ and $\LC(h_1) =a b_1 + b a_1, \LC(h_2) =a b_2 - b a_2$. Since $\vert a_1 \LC(g_1)+a_2 \LC(g_2)\vert>0$ and by the triangle inequality, we obtain
\begin{align*}
&\vert \LC(h_1)\LC(g_1) \vert +\vert \LC(h_2)\LC(g_2) \vert\\
=&\vert (a b_1 + b a_1)\LC(g_1) \vert + \vert (a b_2 - b a_2)\LC(g_2)\vert\\
\geq& \vert  a b_1\LC(g_1)\vert + \vert b a_1 \LC(g_1)\vert +  \vert  a b_2\LC(g_2)\vert + \vert b a_2 \LC(g_2)\vert \\
>& \vert  a b_1\LC(g_1)\vert  +  \vert  a b_2\LC(g_2)\vert \geq  \vert  a b_1\LC(g_1)  +    a b_2\LC(g_2)\vert = \vert ad\vert,
\end{align*}
yielding $ \vert ad\vert < \vert \LC(h_1)\LC(g_1) \vert +\vert \LC(h_2)\LC(g_2) \vert $. Furthermore, we have
\begin{align*}
h_1 g_1 + h_2 g_2=&(\LC(h_1)\LM(h_1)+\mbox{tail}(h_1)) g_1 + (\LC(h_2)\LM(h_2)+\mbox{tail}(h_2))g_2\\
=&(a b_1 + b a_1) \tau \tau_1 g_1 + \mbox{tail}(h_1) g_1+(a b_2 - b a_2) \tau \tau_2 g_2 + \mbox{tail}(h_2) g_2\\
=&  a \tau (b_1 \tau_1 g_1 + b_2 \tau_2 g_2)+b \tau (a_1 \tau_1 g_1 - a_2 \tau_2 g_2) +\mbox{tail}(h_1) g_1+\mbox{tail}(h_2) g_2\\
=&  a \tau \,\textup{gpoly}(g_1, g_2)+b \tau \textup{spoly}(g_1, g_2) +\mbox{tail}(h_1) g_1+\mbox{tail}(h_2) g_2.
\end{align*}
Since the S- and the G-polynomial are of first or second type they reduce to zero w.r.t. $\mathcal{G}$. Hence we can write $h_1 g_1+ h_2 g_2 =\sum_j h_j' g_j$ for $h_j' \in \mathcal{P}^e$ and define $M':=\{ j \in \IN \mid \LM(h_j' g_j)=t \}$. Since $\LM(\tau \textup{spoly}(g_1, g_2))\prec t$, $\LM(\mbox{tail}(h_1) g_1)\prec t$ and $\LM(\mbox{tail}(h_2) g_2)\prec t$ we have
\begin{align*}
    \sum_{j\in M'} \vert \LC(h_j') \LC(g_j) \vert 
=   \vert \LC(a \,\tau \,\textup{gpoly}(g_1, g_2))\vert
=   \vert ad\vert 
<   \vert \LC(h_1)\LC(g_1) \vert +\vert \LC(h_2)\LC(g_2) \vert ,
\end{align*}
which contradicts the assumption that the leading coefficient of our original representation are minimal. Therefore, $M$ contains exactly one element and thus we have a strong Gr\"obner representation of $f$ w.r.t. $\mathcal{G}$, i.e. $\mathcal{G}$ is a strong Gr\"obner basis for $\mathcal{I}$.
\end{proof}
\end{lemma}
\noindent
It is possible to define monic (that is, with
leading coefficients being $1$)
or rather reduced \cite{HLi12, Pauer07} Gr\"obner bases in our setup.
Let $\mathcal{G}\subseteq\mathcal{P}\setminus\{0\}$.
It is called a \textbf{reduced Gr\"obner basis}, if
\begin{enumerate}
    \item every $g\in\mathcal{G}$ has leading coefficient with signum $1$,
    \item $\mbox{L}(\mathcal{G}\setminus\{g\}) \subsetneq \mbox{L}(\mathcal{G})$ for every $g\in \mathcal{G}$ and
    \item $\LT(\tail(g))\notin \mbox{L}(\mathcal{G})$ for every $g\in \mathcal{G}$.
\end{enumerate}
The first condition states, that in the case of $\mathcal{R}=\IZ$ every element of a reduced Gr\"obner basis has leading coefficient in $\IZ_{+}$. The second condition is sometimes referred to as ``simplicity'' and means, that the leading ideal becomes strictly smaller when removing an element, thus no element is redundant. The third condition, ``tail-reduced'', is required in the classical field case with commutative polynomials to ensure that a reduced Gr\"obner basis is unique. However, this does not hold in our situation:
 Pritchard gave the following counterexample in \cite{P96}.

\begin{example}~\\
This example can be used for both the commutative and non-commutative case. Let $f=2y^2$, $g=3x^2+y^2$ and $\mathcal{I}=\langle f, g \rangle\subset \IZ[x,y]$. Then $\{f, g\}$ is a Gr\"obner basis for $\mathcal{I}$ with respect to any ordering $x\succ y$ (however, it is {\bf not} a strong  Gr\"obner basis!) and satisfies the above three conditions for reduced Gr\"obner bases. On the other hand, this is also true for $\{f, g'\}$ where $g'=g-f=3x^2-y^2$, so we have two different reduced Gr\"obner bases for $\mathcal{I}$. In the field case the polynomial $g$ is not tail-reduced.

Consider the same example and compute a strong Gr\"obner basis  in $\IZ[x,y]$ with respect to the degrevlex ordering with $x\succ y$. The result contains just one additional polynomial in both cases, namely 
$x^2 y^2+y^4$ for $\{f, g\}$ 
and 
$x^2 y^2 - y^4$ for $\{f, g'\}$. Unlike the case of fields, this shows that having rings as coefficients leads to non-uniqueness of reduced minimal Gr\"obner bases with normalized coefficients. 
The computations can be done with the following code (note, that the very detailed explanation of the {\sc Singular:Letterplace} usage is done  in the \Cref{ExampleExplained}):
\begin{verbatim}
LIB "freegb.lib"; 
ring r = ZZ,(x,y),dp; // one can use "integer" or ZZ notation for specifying Z
short=0;
option(redSB); option(redTail);
ideal I = 2*y^2, 3*x^2 + y^2;
I = twostd(I); // get new generator x^2*y^2+y^4
ideal J = 2*y^2, 3*x^2 - y^2;
J = twostd(J); // get new generator x^2*y^2-y^4
NF(I,J); // gives 0, that is I is contained in J
NF(J,I); // gives 0, that is J is contained in J
\end{verbatim}
\end{example}
\noindent
When implementing a version of Buchberger's algorithm, one should always aim to have a reduced Gr\"obner basis as an output. In fact this is more practical, because removing elements, which are not simplified or tail reduced speeds up the computation, since we do not need to consider them in critical pairs.

Now we turn our attention to length bounds, needed to certify that
a given finite set of polynomials of length at most $d$ is a Gr\"obner bases subject to a length-compatible monomial ordering. 
In the non-commutative case over fields such a bound is $2d-1$; see e.~g. \cite{LL09}, Cor. 3.19 for the case of a graded ideal, while the extension to the general case is apparent. 


In the initial ISSAC paper \cite{LMZ20} we have formulated the following result as a Lemma, but we have found 
problems with its proof, which require further 
deeper investigations. Therefore, despite the 
nice new example of the bound $3d-1$ in \Cref{LemmaBoundRings} and the evidence from numerous computed examples, we state the following result as a Conjecture.

\begin{conjecture}\label{finiteGB}~\\
Let $\mathcal{G}\subset \mathcal{P}\setminus\{0\}$ be a finite set, which contains polynomials involving monomials of the 
length at most $d\in\IN$. 
Assume moreover, that no new polynomials are added to $\mathcal{G}$, while computing a 
strong Gr\"obner basis with respect to a length-compatible monomial ordering up to length $3d-1$ with the \textsc{BuchbergerAlgorithm}. Then $\mathcal{G}$ is a finite strong Gr\"obner basis for $\langle \mathcal{G}\rangle$.
\end{conjecture}
\noindent
Especially G-polynomials of 2nd kind are potential sources for series of infinitely many elements in Gr\"obner bases with coefficients over rings.

\begin{lemma}~\\
\label{LemmaBoundRings}
In the situation of \Cref{finiteGB}, the seeked bound it at least $3d-1$.
\end{lemma}
\begin{proof}
Consider $\langle 4 x_1 x_2 , x_2 x_3 , 6 x_3 x_4 , x_4 x_1 \rangle$ in $\IZ\langle x_1,x_2,x_3,x_4 \rangle$. 
Since this is a monomial ideal, a monomial 
well-ordering can be chosen freely. We have $d=2$.
In addition to the four original  generators, no new elements appear
during the computation of a Gr\"obner basis up to length $4 = (3\cdot d - 1)-1$. However, from length $5 = 3\cdot d - 1$ there come new generators: $\{ 2 x_1 x_2 x_4 x_3 x_4 , 2 x_1 x_2 x_1 x_3 x_4, 2 x_3 x_4 x_3 x_1 x_2, 2 x_3 x_4 x_2 x_1 x_2 \}$ of length 5, then there are further 14 elements of length 6 and so on.

\begin{verbatim}
LIB "freegb.lib"; 
ring r = ZZ,(a,b,c,d),dp; 
ring R = freeAlgebra(r,6); 
short=0;
option(redSB);
ideal I = 4*a*b, 6*c*d, b*c, d*a;
twostd(I); 
\end{verbatim}

Notably, the same behavior can be already observed with the ideal 
$\langle 2 x_1 x_2 , x_2 x_3 , 3 x_3 x_4 , x_4 x_1 \rangle$.
\end{proof}
\noindent

\section{Coefficient Rings with Zero-divisors}\label{SectionZm}
\noindent
When the ring of coefficients is not a domain like $\IZ$, but an Euclidean {\bf ring} with zero divisors such as $\IZ/m\IZ$ for some nonzero $m\in \IZ$, which is neither a unit nor a prime, then we can make use of factorizations of $m$.

For coefficients of polynomials in $\IZ/m\IZ\langle X \rangle$ there are only two possibilities: They are either units or zero divisors. In the first case, we can treat $\IZ/m\IZ$ like a field. In the second case, one can use a factorization of $m$ into coprime divisors and perform a Gr\"obner basis computation for each divisor, a lifting method. This was done for the commutative case in \cite{EderHofm19} and can be extended as we explain in this section. We will only consider factorizations into coprime numbers and not focus on the case, where $m$ is a prime power. Work on this in the commutative setting was done in \cite[ch. 3 \& 4]{WienandThesis} with an application to modelling fixed bit-width arithmetic and analogies to the commutative case are yet to be investigated.

Recall, that a factorization of $m$, say $m=ab$ for some coprime numbers $a,b\in\IZ$, implies, that $x y\neq m$ for $a\nmid x \mid a$, $ b \nmid y \mid b$. Suppose, that $c x = a, d y=b$ and $x y=m$. Then $m=ab=c x d y=c d m$ and so $m (1-c d)=0$, which implies $1=c d$, because $\IZ$ is a domain. But then $c$ is a unit, which contradicts $a \nmid x$.

This was easy to see, but it also means, that we have to choose our coefficients wisely, when using lifting methods. For $a, b\in \IZ$ coprime, we consider the canonical projections
\begin{align*}
\begin{array}{lrcl} 
& \pi:& \IZ\langle X\rangle & \to\IZ/m\IZ \langle X\rangle, \\
& \pi_a:& \IZ/m\IZ\langle X\rangle \cong (a\IZ+b \IZ)/m\IZ\langle X\rangle & \to \IZ/a\IZ\langle X\rangle\\
\hbox{and} & \pi_b:& \IZ/m\IZ\langle X\rangle \cong (a\IZ+b \IZ)/m\IZ\langle X\rangle & \to \IZ/b\IZ\langle X\rangle. \\
\end{array}
\end{align*}
For an ideal $\mathcal{J}$ of $\IZ/m\IZ\langle X \rangle=: \mathcal{P}_m$, we assume that there exist countable sets $\mathcal{G}_a=\{g_{a, i}\}_i, \mathcal{G}_b=\{g_{b, j}\}_j \subseteq \mathcal{P}_m$, such that $\pi_a(\mathcal{G}_a)\setminus\{0\}$ is a strong Gr\"obner basis for $\pi_a(\mathcal{J})$ and $\pi_b(\mathcal{G}_b)\setminus\{0\}$ is a strong Gr\"obner basis for $\pi_b(\mathcal{J})$. We may demand without loss of generality that $\pi(a) \in \mathcal{G}_a$, $\pi(b)\in \mathcal{G}_b$, since they both map to zero under $
\pi_a$, $\pi_b$ respectively. Furthermore, we assume that $\pi(a)\nmid \lc(g_{a, i}) \mid \pi(a)$ for $g_{a, i}\neq \pi(a)$ and $\pi(b)\nmid \lc(g_{b, j}) \mid \pi(b)$ for $g_{b, j}\neq \pi(b)$. This implies, that each leading coefficient is a nontrivial zero divisor in the respective quotient ring. For every pair $(i, j)$ of indices there exist monomials $\tau_{i, j}, \tau_{j, i}\in \mathcal{B}^e$, such that $\tau_{i, j} \lm(g_{a, i}) = \tau_{j, i} \lm(g_{b, j})$ and one of the four cases
\begin{align*}
\tau_{i, j}= 1 \otimes x' , \tau_{j, i}= y \otimes 1 \hspace{.5cm}& \tau_{i, j}= x \otimes 1 , \tau_{j, i}= 1 \otimes y'\\
\tau_{i, j}= 1 \otimes 1 , \tau_{j, i}= y \otimes y' \hspace{.5cm}& \tau_{i, j}= x \otimes x' , \tau_{j, i}= 1 \otimes 1
\end{align*}
occurs for suitable monomials $x, x', y, y'\in \mathcal{B}$. These are precisely the overlap relations corresponding to first and second type S- and G-polynomials. We define
\begin{align*}
f_{i, j}:=\pi(ar) \lc(g_{a, i}) \tau_{j, i} g_{b, j}+\pi(bs) \lc(g_{b, j}) \tau_{i, j} g_{a, i}
\end{align*}
for a pair $(i,j)$ with $ar+bs=1$.

\begin{theorem}\label{platzhalter}~\\
Let $m=ab\in \IZ$ with $a, b$ coprime such that $ar+bs=1$ for some $r, s \in \IZ$. Furthermore, let $\mathcal{J}$ be an ideal of $\mathcal{P}_m$
accompanied by the sets $\mathcal{G}_a$ and $\mathcal{G}_b$ defined as above.\\
Then $\mathcal{G}:=\{ f_{i, j} \mid \tau_{i, j} \lm(g_{a, i})=\tau_{j, i} \lm(g_{b, j}) \}$ is a strong Gr\"obner basis for $\mathcal{J}$.
\begin{proof} By the second isomorphism theorem we have
\begin{align*}
\begin{array}{ll}
     & (\IZ/m\IZ)/\langle \pi(a)\rangle \cong \IZ/a\IZ \\
\mbox{and}     & (\IZ/m\IZ)/\langle \pi(b)\rangle \cong \IZ/b\IZ.
\end{array}
\end{align*}
From this and the forthcoming \Cref{EderHofmannFactorization1} (after lifting $\mathcal{J}$ to $\mathcal{P}$) it follows that $\mathcal{G}_a \cup \{\pi(a)\} = \mathcal{G}_a$, $\mathcal{G}_b \cup \{\pi(b)\}=\mathcal{G}_b$ are strong Gr\"obner basis of $\mathcal{J}+  \pi(a)\,\mathcal{P}_m$, $\mathcal{J} + \pi(b)\,\mathcal{P}_m$ respectively. Then, after applying the isomorphism theorem one more time, all conditions of the following \Cref{EderHofmannFactorization2} are satisfied and it follows that $\mathcal{G}$ is a strong Gr\"obner basis for $\mathcal{J}$.
\end{proof}
\end{theorem}
\noindent
Note, that the $\tau_{i, j}, \tau_{j, i}$ are not uniquely determined since all overlap relations of the leading monomials have to be considered. The above lemma improves our algorithm for computing strong Gr\"obner bases over principal ideal rings. It remains to show, that the Theorems $10$ and $12$ from \cite{EderHofm19}, formulated in the commutative case also hold in the commutative one.

\begin{theorem}\label{EderHofmannFactorization1} (Generalization of \cite[Theorem 10]{EderHofm19})~\\
Let $m \in \IZ\setminus \{0\}$ and $\mathcal{I}$ an ideal of $\mathcal{P}$. Let $\mathcal{G} \subseteq \mathcal{P}$, such that $\pi (\mathcal{G})$ is a strong Gr\"obner basis of $\pi (\mathcal{I})$. Additionally, we assume that $m \nmid\lc(g) \mid m$ for every $g \in \mathcal{G}$. Then $\mathcal{G} \cup \{m\}$ is a strong Gr\"obner basis for $\mathcal{I} + m\mathcal{P}$.
\begin{proof} Clearly $\mathcal{G} \cup \{m\}$ is a subset of $\mathcal{I} + m\mathcal{P}$. Let $f\in \mathcal{I}$. If $\pi(f)=0$, then $m \mid \lt(f)$. Hence, we may assume without loss of generality that $\pi(f) \neq 0$ and $m \nmid \lc(f)$. Then $\lm(\pi(f))=\lm(f)$ and there exists $g\in \mathcal{G}$ such that $\lt(\pi(g)) \mid\lt(\pi(f))$, because $\pi (\mathcal{G})$ is a strong Gr\"obner basis and we can find a term $h \in \mathcal{P}^e$ with $\pi(h)\lt(\pi(g))=\lt(\pi(f))$. Thus $\lm(h) \lm(g)=\lm(f)$ and $\pi(h\lt(g)-\lt(f))=0$. Thus, we have $h\lt(g)-\lt(f)=c \lm(f)$ for some $c\in m \IZ$ and hence $\lt(g) \mid \lt(f)$, because $\lc(g) \mid m$ by our additional assumption and $\lm(g) \mid \lm(f)$. In other words $\mathcal{G} \cup \{m\}$ is a strong Gr\"obner basis for $\mathcal{I} + m\mathcal{P}$.
\end{proof}
\end{theorem}
\noindent
We intend to use the previous \Cref{EderHofmannFactorization1} in the proof of \Cref{platzhalter} by applying the result to a lift of $\mathcal{J}$.

\begin{theorem}\label{EderHofmannFactorization2} (Generalization of \cite[Theorem 12]{EderHofm19})~\\
Let $\mathcal{J}$ be an ideal of $\mathcal{P}_m$ and $a, b, r, s \in \IZ_m$, such that $ab=0$ and $a, b$ coprime with $ar+bs=1$. Let $\mathcal{G}_a$, $\mathcal{G}_b$ be strong Gr\"obner bases for $\mathcal{J}+ a \mathcal{P}_m$ and $\mathcal{J}+ b \mathcal{P}_m$ respectively, such that for every $g_{a, i} \in \mathcal{G}_a\setminus \IZ_m$ we have $a \nmid\lc(g_{a, i}) \mid a$. Assume, that the same holds for $\mathcal{G}_b$. For $g_{a, i} \in \mathcal{G}_a$ and $g_{b, j} \in \mathcal{G}_b$ we define
\begin{align*}
f_{i, j}:=\pi(ar) \lc(g_{a, i}) \tau_{j, i} g_{b, j}+\pi(bs) \lc(g_{b, j}) \tau_{i, j} g_{a, i}
\end{align*}
and assume $\lc(g_{a, i})\lc(g_{b, j})\neq 0$ for all $i, j$. Then $\mathcal{G}:=\{f_{i, j}\}_{i, j}$ is a strong Gr\"obner basis for $\mathcal{J}$.
\begin{proof} By our assumptions we have $\mathcal{J}=ar \mathcal{J}+ bs \mathcal{J}=ar (\mathcal{J}+b\mathcal{P}_m)+bs( \mathcal{J}+a\mathcal{P}_m)=ar \langle \mathcal{G}_b\rangle +bs \langle \mathcal{G}_a\rangle$. Since $a$ and $b$ are coprime and $\lc(g_{a, i})\mid a$, $\lc(g_{b, j})\mid b$, we see that $\lc(g_{a, i})$ and $\lc(g_{b, j})$ are coprime as well. Furthermore, we have $\lc(g_{a, i})\lc(g_{b, j}) \IZ_m=\lc(g_{a, i})\IZ_m\,\cap\lc(g_{b, j})\IZ_m\supsetneq a\IZ_m \cap b \IZ_m=\{0\}$ and thus $\lt(f_{i, j})=\lc(g_{a, i})\lc(g_{b, j}) \tau_{j,i}\lm(g_{b, j})$. Here we use that the product of leading coefficients is nonzero. Now let $f\in \mathcal{J}\subseteq (\mathcal{J}+ a \mathcal{P}_m) \cap ( \mathcal{J}+ b \mathcal{P}_m)$. Then there exist $g_{a, i}\in \mathcal{G}_a$ and $g_{b, j} \in \mathcal{G}_b$, such that $\lt(g_{a, i}) \mid\lt(f)$ and $\lt(g_{b, j}) \mid\lt(f)$. Especially $\tau_{i,j}\lm(g_{a, i}) \mid \lm(f)$ and $\mbox{lcm}(\lc(g_{a, i}),\lc(g_{b, j})) \mid \lc(f)$. Finally, $\lt(f_{i, j}) \mid \lt(f)$ and $\mathcal{G}$ is a strong Gr\"obner basis for $\mathcal{J}$.
\end{proof}
\end{theorem}

\section{Forming and Discarding Critical Pairs}
\noindent
To improve the procedure \textsc{BuchbergerAlgorithm}, we need criteria to determine which pairs of polynomials of the input set yield S- and G-polynomials, which reduce to zero. 
In the following we will recall the criteria for discarding critical pairs known from the commutative case and analyze, which of them can be applied in the case $\mathcal{R} \langle X \rangle$. 

\begin{remark}~\\
Consider the case $t:=\LM(f)$ is divisible by (or even equal to) $\LM(g)$. Then $\textsc{lcm}(\LM(f), \LM(g))$ contains exactly one element, namely $t$, because it is the only minimal element that is divisible by both leading monomials. Therefore, $\textup{spoly}_1^t(f, g)$ and $\textup{gpoly}_1^t(f, g)$ are the only first type S- and G-polynomials. However, these are not uniquely determined, we might have more overlap relations of $\LM(f)$, $\LM(g)$, as we have seen in the previous example of \Cref{Remark_First_Second_Type_S_And_G_Polynomial}, and we still need second type S-polynomials.
\end{remark}
\noindent
The following Lemma 
explains, in particular, why G-polynomials are redundant over fields.
\begin{lemma} (Buchberger's criterion, generalization of  \cite{EderPfisterPopescu18,Lichtblau12})~\\
Let $f, g \in \mathcal{P}\setminus\{0\}$. If $\LC(f) \mid \LC(g)$ in $\mathcal{R}$, then every G-polynomial of $f$ and $g$ reduces to an S-polynomial of $f$ and $g$.
\begin{proof}
By the hypothesis we have $b=\lcm(\LC(f),\LC(g))=\LC(g)$. Let $r\in \mathcal{R}$, such that $r\LC(f)=\LC(g)$. Then $\LC(f)=(nr+1)\LC(f)-n\LC(g)$ yields any possible B\'{e}zout identity for $b$, where $n\in \IZ$. Thus, with $t=\tau_f \LM(f)=\tau_g \LM(g)$, every G-polynomial of $f$ and $g$ has shape $\textup{gpoly}(f,g)= (nr+1) \tau_f f- n \tau_g g=\LC(f) t + n(r\tau_f \tail(f)-\tau_g \tail(g))+ \tau_f \tail(f)$. Subtracting $\tau_f f$, we can reduce this to $n(r\tau_f \tail(f)-\tau_g \tail(g))$. Note that $r\tau_f \tail(f)-\tau_g \tail(g)$ is an S-polynomial of $f$ and $g$. Hence, every G-polynomial of $f$ and $g$ reduces to zero, after we compute their S-polynomials.
\end{proof}
\end{lemma}
\noindent
For $f\in \mathcal{P}\setminus\{0\}$, we iteratively define $\tail^0(f):=f$ and $\tail^i(f):=\tail(\tail^{i-1}(f))$ for $i\geq 1$.
\begin{lemma}\label{ProductCriterion} (Product criterion, generalization of  \cite{EderPfisterPopescu18,Lichtblau12})~\\
Let $f, g \in \mathcal{P}\setminus\{0\}$ and $w\in \mathcal{B}$, such that
\begin{enumerate}
    \item $\LC(f)$ and $\LC(g)$ are coprime over $\mathcal{R}$,
    \item $\LM(f)$ and $\LM(g)$ only have trivial overlaps and
    \item for all $i,j\geq 1$
    the inequality \; $\LM(\tail^i(f)) w \LM(g) \neq \LM(f) w \LM(\tail^j(g))$ takes place.
\end{enumerate}
Then $s:=\textup{spoly}_2^w(f, g )$ reduces to zero w.r.t. $\{f, g\}$.
\begin{proof}
Under the assumptions \textit{1.} and \textit{2.} we have $s = f w \LT(g) - \LT(f) w g = f w (g-\tail(g))- (f-\tail(f)) w g = \tail(f) w g - f w \tail(g)$. Note that $\tail(f) w g$ reduces to zero w.r.t. $g$ and $f w \tail(g)$ reduces to zero w.r.t. $f$.

By \textit{3.} we can assume without loss of generality that $\LT(s)=\LT(\tail(f)) w \LT(g)$. Then $s$ reduces to $s':=s-\LT(\tail(f)) w g$ and $\LM(s')\prec\LM(s)$. Again by $(3)$ there is no cancellation of leading terms and, since $\prec$ is a well-ordering, we iteratively see that $s$ reduces to zero.
\end{proof}
\end{lemma}
\noindent
The preceding result contains strong conditions to discard S-polynomials. We capture this in the following remark.
\begin{remark}~\\
The commutative version of Buchberger's product criterion in \cite{EderPfisterPopescu18,Lichtblau12} states, that the S-polynomial reduces to zero, if the leading terms are coprime over $\IZ [X]$.

Condition \textit{3.} or rather its negation describes a very specific relation between the terms of $f$ and $g$. There is only a finite amount of $w\in \mathcal{B}$, that satisfy such relation and are at the same time considered in \textsc{BuchbergerAlgorithm}, because we only compute up to a certain length.

The version over fields for this criterion is much simpler, because then we only consider $w$ to be the empty word (which clearly satisfies \textit{3.}). Moreover, \textit{1.} is redundant and Buchberger's product criterion states that an S-polynomial reduces to zero when the leading monomials have only trivial overlap relations.
\end{remark}
\noindent
We give examples for situations in which no possibilities for discarding critical pairs could be found (yet) and which are unique for the commutative case.
\begin{example}~\\
If $\LM(f)$ and $\LM(g)$ have no nontrivial overlap and the leading coefficients are not coprime, i.e. $\gcd(\LC(f), \LC(g))\neq 1$, then we can make no \emph{a priori} statement about reduction. This only applies to second type S- and G-polynomials. Take for example $f=4xy+x, g=6zy+z \in \IZ\langle X \rangle=\IZ\langle x, y, z \rangle$ in the degree left lexicographical ordering with $x \succ y \succ z$. In the sense of \Cref{DefinitionOverlap}, the leading monomials have no nontrivial overlap. Then both
\begin{align*}
\begin{array}{ll}
            & \textup{spoly}_2^1(f, g)=3fzy-2xyg=3xzy-2xyz \\
\mbox{and}  & \textup{gpoly}_2^1(f, g)=(-1)fzy + 1 xy  g=2xyzy+xyz-xzy
\end{array}
\end{align*}
do not reduce any further. Thus, they must be added to the Gr\"obner basis just as any other second type S- and G-polynomial.
Finally, the Gr\"obner basis of $\langle 4xy+x, 6zy+z \rangle$ with respect to classical monomial orderings 
is indeed infinite, because it contains several infinite parametrizable series like $\{ zy^i zy-zy^{i+1} z : i\geq 0 \}$.

\begin{verbatim}
LIB "freegb.lib";
ring r = ZZ,(x,z,y),rp;
ring R = freeAlgebra(r,7);
ideal I = 4*x*y+x, 6*z*y+z;
option(redSB);option(redTail);
ideal J = twostd(I);
\end{verbatim}

When the leading coefficients are not coprime, no statement for S- and G-polynomials of the first type can be made. 
For example, in the case of $f=4xy+y, g=6yz+y$ we have $\textup{spoly}_1^{xyz}(f, g)=3fz-2xg = 3yz-2xy$ and $\textup{gpoly}_1^{xyz}(f, g)=(-1)fz+1xg=2xyz-yz+xy$ which do not reduce any further.

\begin{verbatim}
LIB "freegb.lib";
ring r = integer,(x,y,z),dp;
ring R = freeAlgebra(r,7);
ideal I = 4*x*y+y, 6*y*z+y;
option(redSB);option(redTail);
ideal J = twostd(I);
\end{verbatim}

The Gr\"obner basis of $\langle 4xy+y, 6yz+y \rangle$ with respect to classical monomial orderings seem to be infinite as the one above. This time we see infinite parametrizable series like $\{ y z^i y-y^2 z^{i} : i\geq 0 \}$.
\end{example}

\begin{remark}~\\
In the commutative case, according to \cite{EderPfisterPopescu18}, 
a pair $\{f, g\}$ 
with $\LM(f)=\LM(g)$ can be replaced 
by the new pair $\{\textup{spoly}(f, g), \textup{gpoly}(f, g)\}$. 
Now set $\LM(f)=\LM(g)=:t$, then in the definition of S- and G-polynomials 
of the first type we have $\tau_f=\tau_g=1 \otimes 1$ and therefore $\textup{spoly}_1^t(f, g)=a_f f - a_g g$ and $\textup{gpoly}_1^t(f, g)=b_f f + b_g g$. This yields a linear equation
\begin{align*}
\begin{pmatrix}\textup{spoly}_1^t(f, g) \cr \textup{gpoly}_1^t(f, g) \cr \end{pmatrix}
=
\begin{pmatrix}a_f & -a_g \cr b_f & b_g \cr \end{pmatrix}
\begin{pmatrix}f \cr g \cr \end{pmatrix},
\end{align*}
where the defining matrix has determinant $a_f b_g + a_g b_f=1$, thus it is invertible over $\mathcal{R}$! Hence, we can recover $f$ and $g$ back from their S- and G- polynomials and replace them. The importance of this statement  for the commutative case was discussed in \cite{EderPfisterPopescu18}. Its effectiveness carries over to the non-commutative case.
\end{remark}
\noindent
The following two lemmata are chain criteria, which are based on the idea to have two critical pairs and derive a third one from them under certain conditions. The commutative versions for both criteria were proven in \cite{EderPfisterPopescu18}.

Assertion for both upcoming theorems is the following. Given three polynomials $f, g, h\in \mathcal{P}\setminus\{0\}$, we assume for each pair $p, q \in \{f, g, h\}$, that $\textsc{lcm}(\LM(p), \LM(q))\neq \emptyset$, where $\textsc{lcm}(\bullet,\bullet)$ is defined as in \Cref{DefinitionOverlap}. For each such pair $p,q$ consider monomials $T_{pq}\in\mathcal{B}$ and $\tau_{pq} \in \mathcal{B}^e$ with
\begin{align}\label{BedingungKettenkrit}
\boxed{T_{pq} \in \textsc{lcm}(\LM(p), \LM(q))\hspace{1cm}
\tau_{pq} \LM(p)=T_{pq}\hspace{.2cm} \hspace{1cm}
T_{pq}=T_{qp}.}
\end{align}

\begin{lemma}\label{SChainCrit} (S-chain criterion, generalization of \cite{EderPfisterPopescu18,Lichtblau12})\\
Let $\mathcal{G} \subseteq \mathcal{P}\setminus \{0\}$ and $f, g, h\in \mathcal{G}$ with $\LC(f) \mid \lcm(\LC(g), \LC(h))$ over $\mathcal{R}$. For every pair $p, q \in \{f, g, h\}$ we assume that $\textsc{lcm}(\LM(p), \LM(q))\neq \emptyset$. There exist $T_{pq}\in\mathcal{B}$ and $\tau_{pq} \in \mathcal{B}^e$, such that \Cref{BedingungKettenkrit} holds. Assume that $T_{hg}=T_{gh}$ is divisible by both $T_{hf}$ and $T_{gf}$.\\
If $\textup{spoly}_1^{T_{fg}}(f, g)$ and $\textup{spoly}_1^{T_{fh}}(f, h)$ both have strong Gr\"obner representations w.r.t. $\mathcal{G}$, then so does $\textup{spoly}_1^{T_{gh}}(g, h)$.
\begin{proof}
Let $c_{pq}:=\dfrac{\lcm(\LC(p), \LC(q))}{\LC(p)}$ for $p, q \in \{f, g, h\}$. Then one can check, that
\begin{align*}
	& 	\dfrac{c_{hg}}{c_{hf}} \delta_{gf}\, \textup{spoly}_1^{T_{fh}}(f, h)- \dfrac{c_{gh}}{c_{gf}} \delta_{hf}\, \textup{spoly}_1^{T_{fg}}(f, g)\\
=	&		c_{gh} \delta_{hf} \tau_{gf} g - c_{hg} \delta_{gf} \tau_{hf} h + \left(\dfrac{c_{hg} c_{fh}}{c_{hf}}\delta_{gf} \tau_{fh} - \dfrac{c_{gh} c_{fg}}{c_{gf}} \delta_{hf} \tau_{fg}\right)f.
\end{align*}
Using relations for the monomial expressions $\tau_{pq}, T_{pq}, \delta_{pq}$ and the coefficients $c_{pq}$, we see that the first term on the right hand side is equal to $\textup{spoly}_1^{T_{gh}}(g, h)$ and we obtain
\begin{align*}
\textup{spoly}_1^{T_{gh}}(g, h)=
\dfrac{c_{hg}}{c_{hf}} \delta_{gf}\, \textup{spoly}_1^{T_{fh}}(f, h)- \dfrac{c_{gh}}{c_{gf}} \delta_{hf}\, \textup{spoly}_1^{T_{fg}}(f, g),
\end{align*}
which shows that $\textup{spoly}_1^{T_{gh}}(g, h)$ has a strong Gr\"obner representation w.r.t. $\mathcal{G}$. This works analogously for second type S-polynomials $\textup{spoly}_2^w(g, h)$ or $\textup{spoly}_2^{\tilde{w}} (h, g)$, if we choose $w$ or $\tilde{w}$, such that either $\LM(g) w \LM(h)=T_{gh}$ or $\LM(h) \tilde{w} \LM(g)=T_{hg}$.
\end{proof}
\end{lemma}
\noindent
We give a similar criterion for G-polynomials.
\begin{lemma}\label{GChainCrit} (G-chain criterion, generalization of \cite{EderPfisterPopescu18,Lichtblau12})\\
Let $\mathcal{G} \subseteq \mathcal{P}\setminus \{0\}$ and $f, g, h\in \mathcal{G}$. For every pairs $p, q \in \{f, g, h\}$ let $T_{pq}\in\mathcal{B}$ and $\tau_{pq} \in \mathcal{B}^e$, such that \Cref{BedingungKettenkrit} holds. Additionally we assume, that $\LC(f) \mid \mbox{gcd}(\LC(g), \LC(h))$ and that $T_{hg}=T_{gh}$ is divisible by both $T_{hf}$ and $T_{gf}$.\\
Then $\textup{gpoly}_1^{T_{gh}}(g,h)$ has a strong Gr\"obner representation w.r.t. $\mathcal{G}$.
\begin{proof}
The divisibility condition on the leading coefficient of $f$ yields an element $d\in \mathcal{R}$ with $d\,\LC(f)=\mbox{gcd}(\LC(g), \LC(h))$. Furthermore, there exist $\delta_{gf}, \delta_{hf}\in \mathcal{B}^e$, such that $\delta_{gf} T_{hf}=T_{hg}$ and $\delta_{hf} T_{gf}=T_{gh}$.\\
First observe that
\begin{align*}
     & \textup{gpoly}_1^{T_{gh}}	(g,h)=\mbox{gcd}(\LC(g), \LC(h)) T_{gh}+b_g \tau_{gh}\, \tail(g)+b_h \tau_{hg}\, \tail(h), \\
     & \textup{spoly}_1^{T_{fg}}	(f,g)=	\dfrac{\LC(g)}{\LC(f)}\tau_{fg} f-\tau_{gf} g=\dfrac{\LC(g)}{\LC(f)}\tau_{fg}\,  \tail(f)-\tau_{gf} \, \tail(g) \\
\mbox{and} \hspace{.5cm}  & \textup{spoly}_1^{T_{fh}}	(f,h)=	\dfrac{\LC(h)}{\LC(f)}\tau_{fh} f-\tau_{hf} h=\dfrac{\LC(h)}{\LC(f)}\tau_{fh} \, \tail(f)-\tau_{hf}\,  \tail(h).
\end{align*}
Since $T_{fh}$ divides $T_{gh}$, there exists a $w\in \mathcal{B}^e$ with $w \LM(f) = T_{gh}$ and 
\begin{align*}
w \LM(f)=T_{gh}=\delta_{gf} T_{fh}=\delta_{gf} T_{fh} \LM(f).
\end{align*}
Hence, $w=\delta_{gf} \tau_{fh}$ and analogously $w=\delta_{hf} \tau_{fg}$. Moreover, $dw \LC(f)\LM(f)=\mbox{gcd}(\LC(g), \LC(h)) T_{gh}$ and we obtain
\begin{align*}
&\, \textup{gpoly}_1^{T_{gh}} (g,h) - d w f + b_g \delta_{hf}\, \textup{spoly}_1^{T_{fg}} (f,g) + b_h \delta_{gf} \,\textup{spoly}_1^{T_{fh}}(f,h)\\
=&\, \mbox{gcd}(\LC(g), \LC(h)) T_{gh} - (\mbox{gcd}(\LC(g), \LC(h)) T_{gh} + dw\, \tail(f))\\
&\, + b_g \tau_{gh} \, \tail(g) + b_g \delta_{hf} \left( \dfrac{\LC(g)}{\LC(f)}\tau_{fg} \,  \tail(f)-\tau_{gf} \, \tail(g) \right)\\
&\, + b_h \tau_{hg} \,  \tail(h) +b_h \delta_{gf} \left( \dfrac{\LC(h)}{\LC(f)}\tau_{fh}\, \tail(f)-\tau_{hf} \,  \tail(h) \right)\\
=&\, b_g \tau_{gh}\, \tail(g)+b_h \tau_{hg}\, \tail(h) - dw \, \tail(f)+b_g \dfrac{\LC(g)}{\LC(f)} \delta_{hf} \tau_{fg}\, \tail(f)\\
&\, - b_g \underbrace{\delta_{hf} \tau_{gf}}_{=\tau_{gh}}\, \tail(g) + b_h \dfrac{\LC(h)}{\LC(f)}\underbrace{\delta_{gf} \tau_{fh}}_{=\delta_{hf} \tau_{fg}}\, \tail(f) - b_h \underbrace{\delta_{gf} \tau_{hf}}_{=\tau_{hg}} \, \tail(h)\\
=&\, \left( \dfrac{b_g \LC(g)+\LC(h)}{\LC(f)} \delta_{hf} \tau_{fg}- dw \right) \tail(f)=d (\delta_{hf} \tau_{fg} -w) \, \tail(f)=0.
\end{align*}
Finally, we can write $\textup{gpoly}_1^{T_{gh}}(g,h)$ as
\begin{align*}
\textup{gpoly}_1^{T_{gh}}(g,h)=dwf-b_g \delta_{hf}\, \textup{spoly}_1^{T_{fg}}(f,g)-b_h \delta_{gf}\,\textup{spoly}_1^{T_{fh}}(f,h),
\end{align*}
which is a strong Gr\"obner representation.
\end{proof}
\end{lemma}
\noindent
We conclude, that the well-known criteria for S- and G-polynomials from the commutative case can also be applied in the commu\-ta\-tive case with modifications, if we distinguish between first and second type S- and G-polynomials. Computations show how hard these requirements are to be satisfied compared to the commu\-ta\-tive case by specifically counting the number of applications of product and chain criteria. 
\section{Examples}
\noindent
We give examples for Gr\"obner bases that have been computed up to a certain length bound over the integers. These examples also show that although computing over $\IZ$ delivers infinite results much more often than when computing over fields, commutative Gr\"obner bases over $\IZ$ can be finite as well.

We start with the examples from \cite{Apel2000} until \Cref{Apel4}. Let $\mathcal{P}=\IZ\langle x, y, z \rangle$ with the degree left lexicographical ordering and $ x\succ y \succ z$ (if not indicated otherwise).

\begin{example}\label{ExampleExplained}~\\
We consider the ideal $\mathcal{I}=\langle f_1=yx-3xy-3z, f_2=zx-2xz+y, f_3= zy-yz-x \rangle \subset \mathcal{P}$. 
We investigated it over $\IQ\langle x, y, z \rangle$ in \cite{LSZ20} 
 where we provided detailed comments on syntax and commands of {\sc Singular:Letterplace}. 

\noindent
At first, we analyze this ideal over the field $\IQ$:
\begin{verbatim}
LIB "freegb.lib"; // initialization of free algebras
ring r = 0,(z,y,x),Dp; // degree left lex ord on z>y>x
ring R = freeAlgebra(r,7); // length bound is 7
ideal I = y*x - 3*x*y - 3*z, z*x - 2*x*z +y, z*y-y*z-x;
option(redSB); option(redTail); // for minimal reduced GB
option(intStrategy); // avoid divisions by coefficients
ideal J = twostd(I); // compute a two-sided GB of I
J; // prints generators of J  
\end{verbatim}
The output is a finite Gr\"obner basis
\[
\left\{ 4xy+3z, 3xz-y, 4yx-3z, 2y^2-3x^2, 2yz+x, 
3zx+y, 2zy-x, 3z^2-2x^2, 4x^3+x \right\}.
\]
As we see, original generators have decomposed. In order to compute their expressions in the Gr\"obner basis above, one can use the {\tt lift} command. In particular 
\[
yx - 3xy - 3z = -\frac{3}{4} (4xy+3z) + \frac{1}{4}(4yx-3z).
\]
Now, it seems from the form of leading monomials, that $\IQ\langle x, y, z \rangle/J$ is finite dimensional vector space. Let us check it:
\begin{verbatim}
LIB "fpadim.lib"; // load the library for K-dimensions
lpMonomialBasis(7,0,J); // compute all monomials 
// of length up to 7 in Q<x,y,z>/J
\end{verbatim}
which results in the set $\left\{1, z, y, x, x^2\right\}$, being a monomial $\IQ$-basis of $\IQ\langle x, y, z \rangle/J$.

\newpage

Now, we proceed to work over $\IZ$. For doing this, we need just one change in the code above,
\begin{verbatim}
LIB "freegb.lib"; //initialization of free algebras
ring r = integer,(z,y,x),Dp; //degree left lex ord z>y>x
ring R = freeAlgebra(r,7); // length bound is 7
ideal I = y*x - 3*x*y - 3*z, z*x - 2*x*z +y, z*y-y*z-x;
option(redSB); // Groebner basis will be minimal
option(redTail); // Groebner basis will be tail-reduced
ideal J = twostd(I); // compute a two-sided GB of I
J; // print generators of J  
\end{verbatim}

The output has plenty of elements in each degree (which is the same as length because of the degree ordering), what hints at potentially infinite Gr\"obner basis (what we confirm below) and the elements, which can be subsequently constructed, are 
\begin{align*}
\{ &   f_1, f_2, f_3, 12xy+9z, 9xz-3y, 6y^2-9x^2, 6yz+3x, \\
&   3z^2+2y^2-5x^2, 6x^3-3yz, 4x^2y+3xz, 3x^2z+3xy+3z, \\
&   2xy^2+3x^3+3yz+3x, 3xyz+3y^2-3x^2, 2y^3+x^2y+3xz, \\
&   2x^4+y^2-x^2, 2x^3y+3y^2z+3xy+3z, x^2yz+xy^2-x^3, \\
&   xy^2z-y^3+x^2y, x^5-y^3z-xy^2+x^3, y^3z^2-x^4y, \\
&   x^4z+x^3y+2y^2z+x^2z+3xy+3z, xy^3z-y^4+x^4-y^2+x^2, \\
&   xy^4z-y^5+x^2y^3, xy^5z-y^6+x^4y^2+y^4+x^4+2y^2-2x^2 \}
\end{align*}
Indeed, we can show that $\mathcal{I}$ contains an element with the leading monomial $xy^iz$ for all $i\geq 2$. Therefore this Gr\"obner basis is infinite, but can be presented in finite terms.
Note, that the original generators have been preserved in a Gr\"obner basis, while over $\IQ$ (as above)
 they were decomposed.
Also, over $\IQ$ the input ideal 
has a finite Gr\"obner basis of degree at most 3.
\end{example}

\begin{example}
\label{Torsion}
~\\
Let $\mathcal{I}=\langle f_1=yx-3xy-z, f_2=zx-xz+y, f_3=zy-yz-x \rangle \subset \mathcal{P}$. Then $\mathcal{I}$ has a finite strong Gr\"obner basis, namely
\begin{align*}
\left\{ f_1, f_2, f_3, 8xy+2z, 4xz-2y, 4yz+2x, 
2x^2-2y^2, 4y^2-2z^2, 2z^3-2xy \right\}.
\end{align*}
As we can see, the leading coefficients of the Gr\"obner basis above might vanish, if we pass to the field of characteristic 2. Therefore, the bimodule $M:=\IZ\langle x, y, z \rangle/I$ might have nontrivial 2-torsion, i.e.\ there is a nonzero submodule $T_2(M):=\{ p \in M: \exists \, n\in\IN_0 \; 2^n \cdot p \in I\}$.
In \citep{Hoffmann_Levandovskyy2019}, the classical method of Caboara and Traverso for computing colon (or quotient) ideals has been generalized to non-commutative  case. Using the fact that the ground ring is central (i.e.\ commutes with all variables), we follow that recipe and do the following:
\begin{verbatim}
LIB "freegb.lib"; //we will use position-over-term order
ring r = integer,(x,y,z),(c,dp); 
ring R = freeAlgebra(r,7,2); // 2==the rank of free bimodule
ideal I = y*x - 3*x*y - z, z*x - x*z +y, z*y-y*z-x;
option(redSB); option(redTail); 
ideal J = twostd(I);    module N; 
N = 2*ncgen(1)*gen(1)+ncgen(2)*gen(2),J*ncgen(1)*gen(1);
module SN = twostd(N);   SN;
\end{verbatim}
Above, {\tt gen(i)} stands for the $i$-th canonical basis vector (commuting with everything) and {\tt ncgen(i)} - for the $i$-th canonical generator of the free bimodule, which commutes only with constants. 
The output, which is a list of vectors, looks as follows: 
\begin{verbatim}
...
SN[9]=[0,z*z*z*ncgen(2)-x*y*ncgen(2)]
SN[10]=[2*ncgen(1),ncgen(2)]
SN[11]=[z*y*ncgen(1)-y*z*ncgen(1)-x*ncgen(1)]
...    
\end{verbatim}
\noindent
From this output we gather all vectors with 0 in the first component {\tt ncgen(1)*gen(1)}, and form an ideal of the generators from the second component, whose Gr\"obner basis is
\begin{align*}
\left\{ \,
zy-yz-x, zx-xz+y, yx+xy, \, 2yz+x, 
2xz-y, 2y^2-z^2, 4xy+z, x^2-y^2, \,z^3-xy \, \right\}.
\end{align*}
Another colon computation does not change this ideal, therefore it is the saturation ideal of $I$ at 2, denoted by  $L=I:2^{\infty} \subset \IZ\langle x, y, z \rangle$. It is the presentation for the 2-torsion submodule $T_2(M) =  \IZ\langle x, y, z \rangle L / I$ and, moreover,  
$2 \cdot L \subset I \subset L$ holds.
\end{example}


\begin{example}~\\
\label{Apel4}
In this example we have to run a Gr\"obner basis of $\langle f_1=zy-yz+z^2, f_2=zx+y^2, \, f_3=yx-3xy\rangle$
up to length bound $11$.
We use degree right lexicographical ordering and obtain a finite Gr\"obner basis
\begin{align*}
\{& zy-yz+z^2, zx+y^2 , yx-3xy, 2y^3+y^2z-2yz^2+2z^3,y^2z^2-4yz^3+6z^4,  \\&
y^4+27xy^2z-54xyz^2+54xz^3, \, 54xy^2z-y^3z-108xyz^2+108xz^3+62yz^3-124z^4,\\&
14z^5, 14yz^3-28z^4, 2yz^4-6z^5, 2xyz^3-4xz^4, xy^3z, 2z^6, 2xz^5\} . 
\end{align*}
As we can see from the leading terms, the corresponding module might have $2$- and $7$-torsion submodules.
There have been $17068$ critical pairs created, and internal total length of intermediate elements was $11$. The product criterion has been used $196$ times, while the chain criterion was invoked $36711$ times. Totally, up to $2.9$ GB of memory was allocated.

Comparing the data with increasing the length bound to the presumably unlucky 13, we had to create over $135300$ critical pairs, while the product criterion has been used $1876$ and the chain criterion $365367$ times. This illustrates the explosive behaviour of the number of critical pairs when dealing with rings as coefficients.

In the contrast, the Gr\"obner basis computation of the same input over $\IQ$ considered only $14$ critical pairs, went up to total degree $6$ of intermediate elements, used no product criterion and $9$ times the chain criterion with less than $1$ MB of memory. The result is 
\begin{align*}
\{ & zy-yz+z^2, zx+y^2 , yx-3xy,
2y^3+y^2z-2yz^2+2z^3,\\ & y^2z^2 - 2z^4, xy^2z-2xyz^2+2xz^3, yz^3-2z^4, z^5 \}.
\end{align*}
This demonstrates once again, how technically involved computations with free algebras over rings as coefficients are.
\end{example}
\newpage 
\begin{example}
\label{IwahoriHeckeA3}
~\\
The important class of Iwahori--Hecke algebras \cite{humphreys1990} is associated to Coxeter groups. These algebras are constructed by means of finite presentation over $\IZ[q, q^{-1}]$ where $q$ will later be specialized, most frequently to the root of unity over a finite field. Consider the Iwahori--Hecke algebra of type $A_3$, then it is presented 
as the factor-algebra of
$\IZ[q, q^{-1}] \langle x,y,z \rangle$
modulo the ideal
\[
\langle  
x^2 + (1-q)x - q, y^2 + (1-q)y - q, z^2 + (1-q)z - q, 
zx - xz, yxy - xyx, zyz - yzy
\rangle.
\]
We observe {\bf braid} relations between $x,y$ and $y,z$. In order to treat the ground ring $\IZ[q, q^{-1}]$ appropriately, we do the following: 
\begin{itemize}
\item introduce free variables {\tt q, iq}
where the latter denotes the forthcoming $q^{-1}$,
\item use a block ordering for the variables, giving eliminating preference to the block $x,y,z$,
\item to the ideal of relations above we insert new commutation relations (that is, {\tt q,iq} mutually commute with {\tt x,y,z}) and reciprocity relations for {\tt q,iq}.
\end{itemize}
 
\begin{verbatim}
LIB "freegb.lib";
ring r = integer,(x,y,z,iq,q),(a(1,1,1,0,0),Dp);
ring R = freeAlgebra(r,7);
ideal I = x^2 + (1-q)*x - q, y^2 + (1-q)*y - q, z^2 + (1-q)*z - q, 
z*x - x*z, y*x*y - x*y*x, z*y*z - y*z*y, 
bracket(q,x), bracket(q,y), bracket(q,z), 
bracket(iq,x), bracket(iq,y), bracket(iq,z), q*iq -1, iq*q-1;
option(redSB); option(redTail);
ideal J = twostd(I);
\end{verbatim}

\noindent
The resulting Gr\"obner basis is finite, and consists of the original relations (from the ideal $I$ above) and the only new generator $xyzx-yxyz$ of length $4$. We also observe, that no integers, other than $\pm 1$, appear among 
the coefficients of polynomials from $I$. 
Now we specialize $q$ to the primitive third root of unity.
\begin{verbatim}
ideal L = J, q^2+q+1;
L = twostd(L);
\end{verbatim}

\noindent In the output, as one would expect, only
{\tt q,iq} have been affected. The relation $iq+q+1=0$ has been used to replace {\tt iq} via $-q-1$. 
Since except for the minimal polynomial $q^2+q+1$ and commutativity relations, no $q$ appear as leading coefficients, we 
can proceed to the ground field $\IK:=\IQ[q]/\langle q^2+q+1 \rangle$. One of the possibilities to do this is the localization at $\IZ\setminus\{0\}$. Now, with the abilities of {\sc Letterplace} over fields we easily establish, that specialized over $\IK$, the Iwahori-Hecke algebra of type $A_3$ is finite-dimensional of dimension 24. Hence further computations with modules over this algebra can be carried on.
\end{example}

\begin{example}
\label{Binomial}
~\\
Over $K[X]$, an ideal is called \textbf{binomial}, if it is generated by polynomials of length at most two. A distinct property of binomial ideals, which is easy to prove, is that with respect to any monomial ordering, a binomial ideal possesses a Gr\"obner basis, consisting of binomials. This is not true over rings anymore, as, for instance, a Gr\"obner basis with respect to the degree reverse lexicographical ordering of $\{2x - 3y, xy-3x\}$
is $\{2x-3y,3y^2-9y, xy+x-6y\}$. \\
In the setting of a free algebra, the binomiality of a Gr\"obner basis still holds over $K\langle X\rangle$. As expected, it breaks over rings since in the very same example the commutativity relation $yx-xy$ is a binomial. Hence, a strong 
minimal Gr\"obner basis 
of $\{2x - 3y, xy-3x, yx-xy\} \subset \IZ\langle x, y \rangle$ is
\begin{align*}
\left\{2x-3y,3y^2-9y, xy+x-6y,yx+x-6y
\right\},
\end{align*}
which cannot be made binomial.
\end{example}




\section{Implementation}
\noindent
We have created a powerful implementation called {\sc Letterplace}
\cite{Letterplace2021, LSZ20}
in the framework of {\sc Singular} \cite{singular}.
Its' extension to coefficient rings like $\IZ$ addresses the following functions with the current release for ideals and subbimodules of a free bimodule of finite rank.
We provide a vast family of monomial orderings for ideals and submodules, such as degree right and left lexicographical, including three kinds of orderings, which eliminate variables or free bimodule components.
For modules, position-over-term and term-over-position constructions are available.
\begin{itemize}
    \item {\tt twostd}: a two-sided Gr\"obner basis of a module; when executed with respect to an elimination ordering, it allows to eliminate variables \cite{BoBo}, and thus to compute kernels of ring morphisms and preimages of ideals under such morphisms;
    \item {\tt reduce \text{(or }NF\text{)}}: a normal form of a vector (resp. a polynomial) with respect to a two-sided Gr\"obner basis of a submodule (resp. an ideal);
    \item {\tt syz}: a generating set of a syzygy bimodule \cite{BK07} of a module;
    \item {\tt modulo}: kernel of a bimodule homomorphism;
    \item {\tt lift}: a transformation matrix between a module and a submodule; in other words, {\tt lift} allows to express generators of a submodule in terms of generators of a module;
    \item {\tt liftstd}: (1) a two-sided Gr\"obner basis together with (2) a transformation matrix between the input module and its Gr\"obner basis, and, optionally, (3) a generating set of a syzygy bimodule of the input module; compared to running {\tt twostd, lift, syz} alone, this command does not add much computational overhead.
\end{itemize}
\noindent
{\bf Caveats:} As every software, which is intensively used, 
our implementation has some artefacts, which we cannot overcome and therefore describe as caveats.
\begin{itemize}
\item[a)] Computing with the options {\tt redSB} and {\tt redTail} enabled, sometimes the resulting Gr\"obner basis will not be minimal. This occurs only with rings as coefficients and cannot be changed at this time. Computing 
a Gr\"obner basis of the result one more time produces
a minimal Gr\"obner basis.
\item[b)] A computation, involving Gr\"obner bases, might stop with the following error message:
\begin{verbatim}
? degree bound of Letterplace ring is 9, but at least 
10 is needed for this multiplication   
\end{verbatim}
This is neither a bug nor an error. It indicates that internally a potentially Noetherian reduction has been invoked, what often happens for monomial orderings, which are not compatible with the length of monomials. We recommend to increase the length bound on the ring, and keep polynomials or vectors tail-reduced via {\tt  option(redTail)}.
\item[c)] In \Cref{Torsion} a built-in command {\tt modulo} can be used instead of the construction of the module {\tt SN} and gathering the vectors from the first component. However, because of the degree bound as in b), encountered internally during some multiplication, {\tt modulo} is not coming to a result even after increasing the length bound to high values. Therefore in such cases the explicit construction, like the one of the module {\tt SN} in \Cref{Torsion} will lead to the result.
\end{itemize}

\section{Conclusion and Future Work}
\noindent
Following Mora's ``manual for creating own Gr\"obner basis theory'' \cite{MoraBook4}, 
we have considered the case of free commutative Gr\"obner bases for
ideals and bimodules over $\IZ\langle X \rangle$. We have derived novel information on the building critical pairs and on criteria to discard them when possible. Armed with this
theoretical and algorithmic knowledge,
we have created an implementation in a
{\sc Singular} subsystem {\sc Letterplace}, which offers a rich functionality at a decent speed. We are not aware of yet other systems or packages, which can do such computations. 
 
In this paper we have demonstrated several important applications of
our algorithms and their implementation, in particular the determination of torsion submodules with respect to natural numbers, and operations with Iwahori-Hecke algebras.

A further extension of our implementation to the explicitly given $\IZ/m \IZ$ is planned, along the lines, discussesd in \Cref{SectionZm}. 
Also, we plan to develop (in theory and in practice) one-sided Gr\"obner bases in factor algebras (over fields, {\sc Letterplace} already offers {\tt rightStd} and more functions are under development). 
More functions for dealing with matrices and one-sided modules will make possible the usage of our implementation as a backend from the system {\sc HomAlg} \cite{homalg}. This system performs homological algebra computations within computable Abelian categories
and uses other computer algebra systems as backends for concrete calculations with matrices over rings. Other existing systems like {\sc SageMath} \cite{Sage} and {\sc OSCAR} \cite{OSCAR} can use our implementation as backend, since they have a low-level communication with {\sc Singular}. 
Indeed, thanks to developments in {\sc OSCAR}
in 2019--2021, {\sc Letterplace} functions \cite{Letterplace2021} can be called from and deliver their results
to {\sc OSCAR}. Enhancing this connection is a subject of an ongoing work.

Last but not least, in the opinion of its creators, mathematical software is still not satisfactorily cited. We ask therefore our users to cite our system {\sc Singular:Letterplace} as \citep{Letterplace2021}, which applies both to coefficients over fields and over rings.


\section{Acknowledgements}

The authors are grateful to 
Hans Sch\"onemann, Gerhard Pfister (Kaiserslautern), Leonard Schmitz, Eva Zerz (Aachen), Evelyne Hubert (Sophia Antipolis) and Anne Fr\"uh\-bis-Kr\"uger (Oldenburg) for fruitful discussions. 

We want to thank especially cordially to
 anonymous referees, whose critics and correcting suggestions led to significant enhancement of the readability of this article.

The first and third authors (V.~Levandovskyy and K.~Abou Zeid) have been supported by Project II.6 of SFB-TRR 195 ``Symbolic Tools in Mathematics and their Applications'' of the German Research Foundation (DFG).

The second author (T.~Metzlaff) has been supported by European Union’s Horizon 2020 research and innovation programme under the Marie Sk\l odowska-Curie Actions, grant agreement 813211 (POEMA).

\bibliographystyle{elsarticle-harv}
\bibliography{myref3} 
\end{document}